\documentclass[pdflatex,sn-basic]{sn-jnl}

\normalbaroutside
\usepackage{stmaryrd}
\usepackage{amssymb, amsmath, amsthm}
\usepackage{todonotes}
\usepackage[capitalize]{cleveref}
\usepackage{enumerate}
\usepackage{subfig}
\usepackage{enumitem} 
\usepackage{comment}

\jyear{2022}%

\theoremstyle{thmstyleone}%
\newtheorem{theorem}{Theorem}
\newtheorem{lem}[theorem]{Lemma}
\newtheorem{cor}[theorem]{Corollary}

\theoremstyle{thmstyletwo}%
\newtheorem{remark}{Remark}%

\theoremstyle{thmstylethree}%

\newtheorem{assum}{Assumption}

\newcommand{\la}{\lambda}
\newcommand{\bel}{\begin{equation} \label}
\newcommand{\ee}{\end{equation}}
\def\beq{\begin{equation}}
\def\eeq{\end{equation}}
\newcommand{\jump}[1]{\llbracket#1\rrbracket}
\newcommand{\sdd}{\mathcal{E}}
\renewcommand{\div}{\mathrm{div}\,}  
\def\epsilon{\varepsilon}

\providecommand{\abs}[1]{\left\lvert#1\right\rvert}
\providecommand{\norm}[1]{\left\lVert#1\right\rVert}
\newcommand{\gammaGLS}{\gamma_{\text{GLS}}}
\newcommand{\gammaCIP}{\gamma}
\newcommand{\VC}{V_\mathbb{C}}
\newcommand{\dX}{\mathrm{d}x}
\newcommand{\dS}{\mathrm{d}S}
\renewcommand\Re{\operatorname{Re}}

\begin{document}

\title[Unique continuation for the Lam\'{e} system using stabilized FEM]{Unique continuation for the Lam\'{e} system using stabilized finite element methods}


\author[1]{\fnm{Erik} \sur{Burman}}\email{e.burman@ucl.ac.uk}

\author*[1]{\fnm{Janosch} \sur{Preuss}}\email{j.preuss@ucl.ac.uk}

\affil[1]{\orgdiv{Department of Mathematics}, \orgname{University College London}, \orgaddress{\street{Gower Street}, \city{London}, \postcode{WC1E 6BT},  \country{United Kingdom}}}


\abstract{ 
We introduce an arbitrary order, stabilized finite element method for solving a unique continuation problem subject to the time-harmonic elastic wave equation with variable coefficients. 
Based on conditional stability estimates we prove convergence rates for the proposed method which take into account the noise level and the polynomial degree.
A series of numerical experiments corroborates our theoretical results and explores additional aspects, e.g.\ how the quality of the reconstruction depends on the geometry of the involved domains. We find that certain convexity properties are crucial to obtain a good recovery of the wave displacement outside the data domain and that higher polynomial orders can be more efficient but also more sensitive to the ill-conditioned nature of the problem.
}


\keywords{Lam\'{e} system, Unique continuation, Finite element methods, Conditional H\"older stability}


\pacs[MSC Classification]{35J15, 65N12, 65N20, 65N30, 86-08}

\maketitle

\section{Introduction}\label{sec:intro}
Ill-posed and inverse problems for elastodynamics occur in various geophysical \citep[Section 3]{YYP13} and medical applications \citep{D12}.
Analyzing these problems under realistic assumptions is subject to current research and often requires the use of sophisticated mathematical tools, see e.g.\ \citep{LR00,SUV21,BHKU22}
for the application of microlocal techniques to the recovery of material parameters from certain type of boundary measurements.
Furthermore, numerical methods which make optimal use of the latest analytical results and lead to provably convergent and reliable solutions of the inverse problems 
are in high demand. 
In this paper we aim to present such a method for the unique continuation problem of time-harmonic elastodynamics. 
Arguably, this is the simplest ill-posed problem encountered in this field. 
We consider this problem here because understanding of the stability properties of the continuous problem and the required tools for its numerical treatment are
now sufficiently advanced to allow for a fairly complete convergence analysis.  \par 
The unique continuation problem for the elastic wave equation is formulated as follows.
Let $\Omega \subset \mathbb{R}^d$, $d \geq 2$ be a bounded Lipschitz domain. 
Given $f \in [L^2(\Omega)]^d$ we seek to find the wave displacement $u \in V:= [H^1(\Omega)]^d$ fulfilling
\bel{eq:PDE-cont}
\mathcal{L}u = f  \text{ in } \Omega, 
\ee
where 
\bel{def:Lu} 
\mathcal{L}u := - \nabla \cdot \sigma(u) - \rho u, \quad 
\sigma(u) := 2 \mu \sdd(u) + \lambda \left( \nabla \cdot u \right) I, \quad  
\sdd(u) := \frac{1}{2} \left( \nabla{u} + \nabla{u}^T \right). 
\ee
The Lam\'{e} coefficients $\lambda(x),\mu(x)$ and the density $\rho(x)$ are assumed to be known.
If suitable boundary conditions on $\partial \Omega$ are given, then this problem is well-posed\footnote{Under reasonable regularity assumptions on the data.} and approximate solutions 
of any desired accuracy can be obtained using standard numerical methods.
However, here we will consider the case in which no information of the wave displacement on the boundary is provided. 
To partially compensate for this lack of information, we assume instead that measurements of $u$ in some open subset $\omega \subset \Omega$ are available, that is
\bel{eq:u-measurements-omega}
u  = u_{\omega} \quad  \text{ in } \omega. 
\ee
The objective is then to continue the solution into a larger subset $B \subset \Omega$.  
Note that this problem is ill-posed since continuous dependence on the data fails, i.e.\ an estimate of the form 
$\norm{u}_{B} \leq C  ( \norm{f}_{ \Omega } + \norm{u}_{ \omega} )$, where $B \subset \Omega$ such that $B \setminus \omega \neq \emptyset$,  
is in general not valid.
Here we introduced the shorthand $  \norm{f}_{ M } := \norm{f}_{ [L^2(M)]^d } $ for a subset $M \subset \mathbb{R}^d$.  
It is possible though to obtain (see \citep{LNW10,LNUW10} and \cref{section:cond-stability} for details) a \textit{conditional stability estimate} of the form 
\begin{equation}\label{eq:cond-stability-intro}
\norm{u}_{B} \leq C \left( \norm{f}_{ \Omega } +  \norm{u}_{\Omega}   \right)^{1-\tau} \left( \norm{f}_{ \Omega } + \norm{u}_{ \omega }  \right)^{\tau}
\end{equation}
on a subset $\omega \subset B \subset \Omega$ such that $B \setminus \omega$ does not touch the boundary of $\Omega$.
Note that the first factor in \cref{eq:cond-stability-intro} involves $\norm{u}_{\Omega}$ and that the ill-posedness of the problem increases with decreasing H{\"o}lder exponent $\tau \in (0,1)$. 
\par
A lack of well-posedness precludes the use of many established numerical methods (e.g. standard finite elements) which heavily rely on this property to obtain reliable approximate solutions. 
In order to apply these methods anyway, the continuous problem is usually approximated by a series of well-posed problems which are perturbations of the original problem. 
We will follow a different approach here based on casting the original data assimilation problem as a constrained optimization problem at the discrete level. 
This discrete problem is unstable since no regularization has been introduced at the continuous level. 
Subsequently, regularization will be added at the discrete level by utilizing stabilization terms well-known in the finite element community.
An appropriate choice thereof allows us to conduct an error analysis which exploits the conditional stability estimate of \cref{eq:cond-stability-intro}
and leads to explicit convergence rates (see \cref{thm:L2-error-perturbed-data}). \\ 
Our method is based on a general framework for noncoercive problems that has been introduced by \citep{B13} and was thereupon applied to a variety of problems including unique continuation and source reconstruction for the Poisson problem \citep{EHL18} as well as data assimilation for the heat \citep{BO18} and linearized Navier-Stokes equations \citep{BBFV20}.
Concerning time-harmonic wave equations, \citep{N20} treated unique continuation for the constant coefficient Helmholtz equation in his dissertation using piecewise affine finite elements, see also \citep{BNO19}. A hybridized high order method for the same problem has been analyzed in \citep{BDE21}.
In relation to the literature, the contributions of the paper on hand are as follows:
\begin{itemize}
\item We generalize the method from \citep{BNO19} to the case of elastic wave propagation, in particular we treat the Lam\'{e} system instead of the scalar Helmholtz equation.
\item Additionally, we carry out an error analysis for arbitrary polynomial orders and investigate the benefits of using higher order polynomials in numerical experiments. 
In contrast to the hybrid high order method presented in \cite{BDE21}, standard $H^1$-conforming finite elements are employed in this work. 
\item Whereas the publications \citep{BNO19,BDE21} treat the case of constant coefficients, we allow for a spatial dependence of the material parameters and present numerical experiments for the practically relevant setting of a jumping shear modulus. The shortcoming for working at this level of generality is that in contrast to the cited works our error analysis is not explicit in the wavenumber.  
\end{itemize}
The remainder of this paper is structured as follows.
In \cref{section:cond-stability} we give a precise statement of the conditional stability estimate of \cref{eq:cond-stability-intro}, whose 
actual derivation is deferred to \cref{section:Continuum_stability_estimate_derivation}. 
In \cref{section:Discretisation} we introduce a stabilized finite element method to numerically approximate the unique continuation problem from \cref{eq:PDE-cont}-\cref{eq:u-measurements-omega}.
\Cref{section:ErrorAnalysis} presents an analysis which leads to $L^2$-error estimates first for the case of unperturbed (\cref{thm:L2-error-unperturbed}) 
and then for perturbed data (\cref{thm:L2-error-perturbed-data}).
Numerical experiments that confirm our theoretical findings and investigate additional aspects are presented in \cref{section:numexp}. 
We finish with a conclusion and an outlook towards future research. 

\section{Conditional stability result for the continuous problem}\label{section:cond-stability}
Deriving unique continuation or conditional stability results for the Lam\'{e} system requires some regularity assumptions on the coefficients.
The foundation for the conditional stability estimate employed in this paper is a three ball inequality derived in \citep{LNUW10} which is based on 
the following assumption. 
\begin{assum}\label{ass:three_ball_ieq}
Let $\mu \in C^{0,1}(\Omega)$, and let $\lambda,\rho \in L^{\infty}(\Omega)$ satisfy
\bel{eq:assumption-coeff}
\left\{ \begin{array}{rcll} & \mu(x) \geq \delta_0, &  \lambda(x) + 2 \mu(x) \geq \delta_0 > 0  \quad  \forall x \in \Omega,  \\
&  \norm{\mu}_{C^{0,1}(\Omega)} +  \norm{\la}_{L^{\infty}(\Omega)} \leq M_0,  &  \norm{\rho}_{L^{\infty}(\Omega)}  \leq M_0   . \end{array}\right.
\ee
for some positive constants $\delta_0$ and $M_0$. Here, 
\[ \norm{g}_{C^{0,1}(\Omega)}  := \norm{g}_{L^{\infty}(\Omega) } + \norm{ \nabla g }_{L^{\infty}(\Omega) }. \]
\end{assum} 
The three ball inequality of \citep[Theorem 1.1]{LNUW10} takes the following form.
\begin{theorem}\label{thm:three-ball-ieq-homog}
Let the origin of $\mathbb{R}^d$ be contained in $\Omega$.
There exists a positive number $\tilde{R} < 1$, depending only on $d,M_0,\delta_0$, such that if 
\[ 0 < R_1 < R_2 < R_3 \leq R_0  \text{ and }  R_1/R_3 < R_2 / R_3 < \tilde{R}, \]
then 
\bel{eq:three-ball-ieq-homog}
\int\limits_{ \abs{x} < R_2 } \abs{u}^2 \mathrm{d}x  \leq 
C \left( \int\limits_{ \abs{x} < R_1 } \abs{u}^2 \mathrm{d}x \right)^{\tau}  
\left( \int\limits_{ \abs{x} < R_3 } \abs{u}^2 \mathrm{d}x \right)^{1 - \tau}  
\ee
for $u \in H^{1}_{\text{loc}}(B_{R_0})$ satisfying $\mathcal{L}u = 0$ in $B_{R_0}$, where the constant $C$ depends on $R_2 / R_3, d,M_0,\delta_0$, 
and $0 < \tau < 1$ depends on $R_1/R_3, R_2 / R_3, d, M_0, \delta_0$.
Moreover, for fixed $R_2$ and $R_3$, the exponent $\tau$ behaves like $1/(-\log R_1 )$ where $R_1$ is sufficiently small.
\end{theorem}
To obtain a conditional stability result from \cref{thm:three-ball-ieq-homog} that is suitable for our purpose, 
we require well-posedness of the interior impedance problem. 
Let us fix some notation before stating the required result. 
\begin{itemize}
\item As in the introduction let $V:= [H^1(\Omega)]^d$ denote the usual Sobolev space of real-valued functions with square integrable weak derivatives up to first order and 
$V_0:= [H^1_0(\Omega)]^d$ denote $V$ with homogeneous Dirichlet boundarty conditions included. 
		We use a prime to denote the corresponding dual spaces, i.e.\ $V^{\prime}$ and $V_0^{\prime}$.
\item  Let $\VC$ denote the Sobolev space $[H^1(\Omega)]^d$ of functions taking values in the complex numbers with inner product 
$(u,v)_{\VC} := \int\limits_{\Omega} ( \nabla u \nabla \bar{v} +  u \bar{v} ) \; \dX$ and $\VC^{\prime}$ denote its dual space.
The corresponding space with homogeneous Dirichlet boundary conditions and its dual are denoted by $V_{\mathbb{C},0}$ and $V_{\mathbb{C},0}^{\prime}$, respectively.
\end{itemize}
\begin{assum}\label{ass:IIP_Robin}
Let $\Omega$ be a bounded Lipschitz domain, $f \in \VC^{\prime}$ and $k > 0$. 
We assume that there exists a unique solution $ u \in \VC $ of the problem 
\bel{eq:PDE-IIP}
\left\{ \begin{array}{rcll} &\mathcal{L} u &=  f & \text{ in } \Omega ,\\
&  \sigma(u) \cdot \mathbf{n}_{\partial \Omega} + i ku & = 0 \quad & \text{ on } \partial \Omega,  \end{array}\right.
\ee
fulfilling the stability bound
\bel{eq:stability_IIP}
\norm{u}_{\VC} \leq C \norm{f}_{\VC^{\prime}}.
\ee
Here, $\mathbf{n}_{ \partial \Omega}$ denotes the exterior normal vector on $\partial \Omega$. The constant $C$ 
is assumed to be independent of $u$ and $f$ but may depend on the material parameters, $k$ and the domain $\Omega$.
\end{assum}
If $\partial \Omega$ and the Lam\'{e} coefficients are sufficiently smooth, then \cref{ass:IIP_Robin} follows by exploiting elliptic regularity. 
Indeed, as Korn's inequality and the assumption $\lambda(x) + 2 \mu(x) > 0 $ yield a G{\aa}rding inequality, the Fredholm alternative 
implies the desired well-posedness provided that uniqueness can be shown. 
To this end, note that a solution of \cref{eq:PDE-IIP} with $f=0$ has to vanish on $\partial \Omega$ which follows by taking the imaginary part of the weak 
formulation using $\bar{u}$ as a test function. If one can now show that $\sigma(u) \cdot \mathbf{n}_{\partial \Omega}$ vanishes as well on $\partial \Omega$, 
then $u$ can be extended by zero to an $H^1$-solution in all of $\mathbb{R}^d$ which implies by \cref{thm:three-ball-ieq-homog} that it must vanish everywhere.
At this point, smoothness assumptions are required to obtain that $\mathcal{L}u = 0$ in $\Omega$ which implies vanishing of $\sigma(u) \cdot \mathbf{n}_{\partial \Omega}$ 
on the boundary using integration by parts. \par 
For applications to high-frequency wave propagation it is important to understand how the constant $C$ in the stability bound given in \cref{eq:stability_IIP} depends on 
the coefficient $\rho$ in \cref{def:Lu}. 
According to the next lemma, for a homogeneous medium the dependence is fortunately no worse than linear. 
\begin{lem}\label{lem:IIP-Robin}
Let $\Omega$ be a bounded Lipschitz domain, $f \in \VC^{\prime}$ and $k>0$. 
If $\mu, \lambda$ and $\rho=k^2$ are constant with $k \geq 1$ and $d=3$, then
\begin{equation}\label{eq:stability_IIP_k}
\norm{\nabla u}_{ \Omega } + k \norm{ u}_{ \Omega } \leq C k^2  \norm{f}_{\VC^{\prime}},  
\end{equation}
for the solution $u$ of \cref{eq:PDE-IIP} holds with $C$ being independent of $k$.
\end{lem}
\begin{proof}
Given in \cref{section:Continuum_stability_estimate_derivation}.
\end{proof}
Let us remark that \cref{lem:IIP-Robin} is obtained as a Corollary from \citep[Theorem 2.7]{BG16} in which the authors proved
\cref{eq:stability_IIP_k} with $\norm{f}_{ \Omega }$ on the right hand side and a factor of $k$. 
Actually, a sharper bound which is $\mathcal{O}(1)$ in $k$ has been obtained in \citep[Proposition 4.3]{CFN19} by imposing stronger smoothness assumptions on $\partial \Omega$ and requiring that $\Omega$ is star-sharped. 
Hence, under these additional assumptions the bound in \cref{eq:stability_IIP_k} could be lowered from $k^2$ to $k$.  \par 
Utilizing well-posedness of \cref{eq:PDE-IIP} allows to mold the three-ball inequality of \cref{thm:three-ball-ieq-homog} into a form which is suitable for the numerical analysis in \cref{section:ErrorAnalysis}. 
\begin{cor}\label{cor:cond-stability}
Let $u$ be a solution of $\mathcal{L} u = f \in V_0^{\prime}$.
Consider subdomains $\omega \subset B \subset \Omega$ such that $B \setminus \omega$ does not touch the boundary of $\Omega$.  
Then there exists a constant $C>0$ and $\tau \in (0,1)$ such that
\begin{equation}\label{eq:cond-stability}
\norm{u}_{L^2(B)} \leq C \left( \norm{f}_{ V_{0}^{\prime} } +  \norm{u}_{[L^2(\Omega)]^d}   \right)^{1-\tau} \left( \norm{f}_{ V_{0}^{\prime} } + \norm{u}_{ [L^2(\omega)]^d}  \right)^{\tau}.
\end{equation}
\end{cor}
\begin{proof}
Given in \cref{section:Continuum_stability_estimate_derivation}.
\end{proof}

\section{Discretisation}\label{section:Discretisation}

In this section we introduce a stabilized finite element method to numerically approximate the unique continuation problem given in \cref{eq:PDE-cont}-\cref{eq:u-measurements-omega}.
In \cref{ssection:FEspaces} triangulations and finite element spaces are defined. 
We proceed in \cref{ssection:Lagrangian} by defining a Lagrangian functional from which numerical approximations to the wave displacement will be obtained as saddle points.
In \cref{ssection:Stab} we specify stabilization terms for this Lagrangian.
Suitable norms for the error analysis are presented in \cref{ssection:Norms}.
Interpolation operators and certain stability estimates for them are considered in \cref{ssection:Interp}.

\subsection{Finite element spaces}\label{ssection:FEspaces}
For the analysis it will be assumed that the domain $\Omega$ is polygonal. 
This is consistent with the regularity requirement on $\Omega$ stated in \cref{ass:IIP_Robin}.
Consider then a family $ \mathcal{T} = \{ \mathcal{T}_h \}_{ h >0}$ of triangulations of $\Omega$ consisting of simplices $K \in \mathcal{T}_h$ such that the intersection of any two 
distinct ones is either a common vertex, a common edge or a common face. 
Further assume that the family $\mathcal{T}$ is quasi-uniform and fitted to the subsets $\omega$ and $B$.
Let $\mathcal{F}_i$ denote the set of all interior facets of the triangulation.
Let $X_h^{p}$ be the standard $H^1$-conforming finite element space of piecewise polynomials of order $p$ on $\mathcal{T}_h$. 
Define 
\bel{eq:FEM-spaces}
V_h^p:= [X_h^p]^d, \qquad W_h^p := V_h^p \cap V_0. 
\ee
For ease of notation the superscript $p$ will usually be omitted below.

\subsection{Lagrangian and optimality conditions}\label{ssection:Lagrangian}

The weak formulation of the partial differential equation (PDE) constraint in \cref{eq:PDE-cont} is given by: Find $u \in V$  such that $a_h(u,v) = (f,v)_{\Omega} $ for all $v \in V_0$, where 
\bel{eq:blf}
a_h(u,v) := \int\limits_{\Omega} \left[ \sigma(u):\sdd(v) - \rho uv  \right] \mathrm{d}x, 
\ee
and $(v,w)_{M} := (v,w)_{ [L^2(M)]^d }$ for any $M \subset \mathbb{R}^d$.
Following \citep{BNO19,N20} we define the Lagrangian 
\begin{align}
\begin{aligned}	\label{eq:Lagrangian}
L(u_h,z_h) &=  \frac{1}{2} \norm{u_h - u_{\omega} }_{ \omega }^2 + \frac{1}{2} s_{\gamma}(u_h-u,u_h-u) + \frac{1}{2} s_{\alpha}(u_h,u_h)  \\ 
	& - \frac{1}{2} s^{\ast}(z_h,z_h)  + a_h(u_h,z_h) +  s_{\beta}(u_h,z_h) - (f,z_h)_ {\Omega} 	
\end{aligned}
\end{align}
which contains aside from the data fidelity and PDE constraint four stabilization terms $s_{\gamma},s_{\alpha}$,$s_{\beta}$  and $s^{\ast}$ that will be specified later.
Since a solution $u$ of \cref{eq:PDE-cont} is explicitly inserted in $s_{\gamma}$ above, we have to be careful in choosing this stabilization so that it 
can indeed be implemented using only the given data $f$ and $u_{\omega}$. We will see below in \cref{eq:s_gamma_u_vh_known} that this is indeed the case. \par 
The first order optimality conditions lead to the equations: Find $(u_h,z_h) \in V_h \times W_h$ such that  
\bel{eq:opt-cond}
\begin{array}{rcll} & (u_h,v_h)_{ \omega } + s_{\gamma}(u_h-u,v_h) + s_{\alpha}(u_h,v_h) + a_h(v_h,z_h) +  s_{\beta}(v_h,z_h)  &=  (u_{\omega},v_h)_{ \omega },  \\
 & a_h(u_h,w_h) - s^{\ast}(z_h,w_h) +  s_{\beta}(u_h,w_h)    &= (f,w_h)_{ \Omega }     \end{array}
\ee
for all $(v_h,w_h) \in V_h \times W_h$.
This can be written in the compact form: 
Find $(u_h,z_h) \in V_h \times W_h$ such that  
\begin{align}
\begin{aligned}	\label{eq:opt-compact}
A[(u_h,z_h),(v_h,w_h)] &= (u_{\omega},v_h)_{ \omega } + s_{\gamma}(u,v_h) + (f,w_h)_{ \Omega } \quad \forall (v_h,w_h) \in V_h \times W_h, 
\end{aligned}
\end{align}
with 
\begin{align}
\begin{aligned}	\label{eq:def-A}
A[(u_h,z_h),(v_h,w_h)] :&= (u_h,v_h)_{ \omega } + s_{\gamma}(u_h,v_h) + s_{\alpha}(u_h,v_h) + a_h(v_h,z_h)   \\
                        & +  s_{\beta}(v_h,z_h) - s^{\ast}(z_h,w_h) + a_h(u_h,w_h) +  s_{\beta}(u_h,w_h).   
\end{aligned}
\end{align}

\subsection{Stabilization}\label{ssection:Stab}
In this section we introduce suitable stabilization terms which are crucial for our method to operate properly. 
For well-definedness and consistency of the stabilization some regularity assumptions on the coefficients 
are required. 
To keep the exposition clear, we will first introduce all stabilization terms in \cref{sssection:stab_def} under the assumption of smooth coefficients 
and then specify in \cref{sssection:stab_reg} the minimal regularity requirement under which specific parts of the stabilization can be activated. 

\subsubsection{Definition}\label{sssection:stab_def}  
We start by introducing a notation for the $j$-th order jumps of the stress $\sigma(u)$ in normal direction over the interior facets:
\bel{eq:CIP-full-gradient}
J_j(u_h,v_h) := \sum\limits_{F \in \mathcal{F}_i} \int\limits_{F} h^{2j-1} \jump{  (\nabla^{j-1} \sigma(u_h)) \cdot \mathbf{n}  } \jump{ (\nabla^{j-1} \sigma(v_h)) \cdot \mathbf{n}  } \; \mathrm{d}S, \quad j \geq 1.  
\ee
Here, the jump over a facet $F = K_1 \cap K_2$ for two neighboring simplices $K_1,K_2 \in \mathcal{T}_h$ is defined as 
\bel{eq:facet_jump_def}
\jump{ (\nabla^{j-1} \sigma(u_h))  \cdot \mathbf{n} } := \left.(\nabla^{j-1} \sigma(u_h))\right|_{K_1} \cdot \mathbf{n}_1 + \left.(\nabla^{j-1} \sigma(u_h))\right|_{K_2} \cdot \mathbf{n}_2,
\ee
where $\mathbf{n}_i$ are the outward pointing normal vectors of $K_i,i=1,2$.
These jump terms appear in both stabilizers $s_{\beta}$ and $s_{\gamma}$, albeit are applied to different variables and fulfill a separate purpose. 
The stabilizer
\bel{eq:s_beta-def}
s_{\beta}(u_h,w_h) := \sum\limits_{j=1}^{p} \beta_j J_j(u_h,w_h) 
\ee
for penalty parameters $\beta_j \in \mathbb{R}, j=1,\ldots,p$, represents a perturbation of the original PDE constraint whose effectiveness to mitigate 
polution effects will be investigated in numerical experiments (see \cref{ssection:numexp_pollution}).  
This is inspired by the well-known continuous interior penalty (CIP)-FEM for the Helmholtz equation, see e.g. \citep{W13,ZW13,DH15,ZW22}. 
Let us mention that this stabilization term is optional in the sense that the final error estimate stated in \cref{thm:L2-error-perturbed-data} holds even for $\beta_j = 0, j=1,\ldots,p$.
\par 
Jump terms also appear as part of the stabilizer 
\bel{eq:s_gamma-def}
s_{\gamma}(u_h,v_h) := \sum\limits_{j=1}^{p}  \gammaCIP_{j} J_j(u_h,v_h) + \gammaGLS h^2 (\mathcal{L}u_h,\mathcal{L} v_h)_{ \mathcal{T}_h },
\ee
which is required to guarantee unique solvability of \cref{eq:opt-compact}.
Here, 
\[ h^2 (\mathcal{L}u_h,\mathcal{L} v_h)_{ \mathcal{T}_h } := h^2 \sum\limits_{K \in \mathcal{T}_h} (\mathcal{L}u_h,\mathcal{L} v_h)_{K},
\] 
is a Galerkin least squares stabilization. 
We will require that the penalty parameters satisfy
\bel{eq:penalty-param-assumption}
\gammaCIP_1 > 0 \text{ and } \gammaCIP_j \geq \max\{0,\abs{\beta_j}\} \text{ for } j=2,\ldots,p  \text{ and } \gammaGLS > 0.
\ee
Similar as in \citep{BDE21},  we additionally add a discrete Tikhonov regularization term
\bel{eq:s_alpha-def}
s_{\alpha}(u_h,v_h) := \alpha h^{2p} (u_h,v_h)_{ \Omega },
\ee
for some $\alpha >0$ to control the $L^2$-norm of the approximation $u_h$ on all of $\Omega$.  \par
The stabilization for the dual variable is defined as 
\bel{eq:s_star_h-def}
s^{\ast}(z_h,w_h) := \int\limits_{\Omega} \nabla z_h : \nabla w_h \; \mathrm{d}x. 
\ee

\subsubsection{Regularity requirements}\label{sssection:stab_reg}  
We will propose a numerical method that is well-defined for jumping shear moduli as they occur in practical applications, e.g.\ in seismology.
Even though such jumps will violate the regularity \cref{ass:three_ball_ieq} for the three ball inequality of \cref{thm:three-ball-ieq-homog} on which our error estimates will be based,
we can nevertheless implement our method and carry out numerical experiments if jumps occur, see \cref{ssection:numexp_jump_shear}. \par
In the Galerkin least squares stabilization given in \cref{eq:s_gamma-def} the strong form of the differential operator $\mathcal{L}$ is applied in an element-wise fashion.
This requires the following assumption on the Lam\'{e} coefficients.
\begin{assum}\label{ass:lame_reg_FEM}
	We will assume that the meshes $\mathcal{T}_h$ can be constructed so that possible singularities of $\mu$ and $\lambda$ only occur on element edges (for $d=2$) or faces (for $d=3$), that is 
	$\mu, \lambda \in H^1(\mathcal{T}_h)$, where
\[ 
H^1(\mathcal{T}_h) := \{ v \in L^2(\Omega) \mid \forall K \in \mathcal{T}_h, v|_{K} \in H^1(K)  \}. 
\]
\end{assum}
This assumption appears to be realistic for applications in global seismic wave propagation in which meshes are usually contructed to respect 
the singularities of stratified reference earth models, see e.g.\ \citep[Figure 6]{KT02}. \par
Next we will discuss the other contribution in \cref{eq:s_beta-def} and \cref{eq:s_gamma-def}, i.e.\ the jump terms over the facets.  
The analysis presented in \cref{section:ErrorAnalysis} requires that these terms are consistent, i.e.\ if $u$ is a weak solution of 
\cref{eq:PDE-cont}-\cref{eq:u-measurements-omega}, then 
\bel{eq:CIP_consistency}
\gammaCIP_{j} J_j(u,v) = \beta_{j} J_j(u,v) =  0 \text{ for } j=1,\ldots,p  
\ee
and $v$ in $V+V_h$ has to hold.
Since the meshes are assumed to be aligned with discontinuities of the Lam\'{e} coefficients, this automatically holds for $j=1$ as weak solutions 
are required to satisfy $\jump{  \sigma(u) \cdot \mathbf{n}  } = 0 $ across an interface over which the coefficients exhibit jumps. 
However, higher order jumps do not need to vanish and so it is not conducive to penalize them. 
Therefore, we will set $\beta_j$ and $\gammaCIP_{j} $  to zero for $j \geq 2$ unless the Lam\'{e} coefficients are smooth. 
Note that this is consistent with \cref{eq:penalty-param-assumption}. 
\begin{assum}\label{ass:stab_param_smoothness_coeff}
If $\mu, \lambda \notin C^{\infty}(\Omega)$ then $\beta_j = \gammaCIP_{j}  = 0$ for $j \geq 2$. 
\end{assum}
Under \cref{ass:stab_param_smoothness_coeff} we have that \cref{eq:CIP_consistency} holds in any case since either the corresponding penalty parameter vanishes
or because $\jump{  (\nabla^{j-1} \sigma(u)) \cdot  \mathbf{n}  } = 0$ for a sufficiently regular solution $u \in [H^{p+1}(\Omega)]^{d}$,
see e.g. \citep[Lemma 1.23]{DPE12}.
A possibility to relax \cref{ass:stab_param_smoothness_coeff} could be to allow for spatially-varying penalty parameters which vanish in regions where the material parameters are non-smooth and may take positive values elsewhere.
We conclude this subsection by noting that
\bel{eq:s_gamma_u_vh_known}
s_{\gamma}(u,v_h) = \gammaGLS h^2 (\mathcal{L}u,\mathcal{L} v_h)_{ \mathcal{T}_h } = \gammaGLS h^2 (f,\mathcal{L} v_h)_{ \mathcal{T}_h  },
\ee
holds, which shows that the right hand side of \cref{eq:opt-compact} is known.
\subsection{Norms and inf-sup condition}\label{ssection:Norms} 
We define 
\bel{eq:norms-V_h-W_h}
\norm{u_h}_{V_h} := \left( s_{\gamma}(u_h,u_h) + s_{\alpha}(u_h,u_h)  \right) ^{1/2},   \quad 
\norm{z_h}_{W_h} := s^{\ast}(z_h,z_h)^{1/2},
\ee
for $u_h \in V_h$ and $z_h \in W_h$. 
Since $\alpha >0$ and thanks to the Friedrichs inequality on $W_h \subset V_0$, c.f. \cref{ieq:Friedrich}, 
these expressions indeed define norms on $V_h$, respectively $W_h$. 
On the product space $V_h \times W_h$ we define
\bel{eq:norm-V_hxW_h}
\norm{(u_h,z_h)}_{s}^2 := \norm{u_h}^2_{V_h} + \norm{u_h}^2_{ \omega } + \norm{z_h}_{W_h}^2, \quad (v_h,z_h) \in V_h \times W_h 
\ee
which then also defines a norm on $V_h \times W_h$. 
Note that 
\beq
A[(u_h,z_h),(u_h,-z_h)] = \norm{u_h}_{ [L^2(\omega)]^d}^2 + \norm{u_h}_{V_h}^2 + \norm{z_h}_{W_h}^2  = \norm{ (u_h,z_h)}_s \norm{(u_h,-z_h)}_{s}, 
\ee
which implies the inf-sup condition 
\bel{eq:inf-sup}
\sup_{ (v_h,w_h) \in V_h \times W_h} \frac{ A[(u_h,z_h),(v_h,w_h)]   }{ \norm{ (v_h,w_h)}_s   } \geq C \norm{ (u_h,z_h) }_s.
\ee
Here and in the following $C>0$ denotes a generic constant independent of $h$ but possibly depending on the stabilization parameters.  
\subsection{Interpolation}\label{ssection:Interp}
Let $\Pi_h:V \rightarrow V_h$ denote the Scott-Zhang interpolation operator, which preserves homogeneous boundary conditions and fulfills (see \cite{SZ90}) the following stability  
\bel{eq:H1-stability-SZ}
\norm{\Pi_h u}_{ [H^1(\Omega)]^d } \leq C \norm{u}_{[H^1(\Omega)]^d}, \quad \forall u \in [H^1(\Omega)]^d 
\ee
and approximation property:
\bel{eq:approx-interp-ho}
\norm{u - \Pi_h u}_{[H^m(\Omega)]^d}  \leq C h^{s-m} \norm{u}_{[H^{s}(\Omega)]^d}, \; \forall u \in [H^s(\Omega)]^d, 
\ee
with $1 \leq s \leq p+1$ and $0 \leq m \leq s$. \par 
We will now derive some further approximation and stability results for this interpolation required for the analysis in \cref{section:ErrorAnalysis}.
To this end, the following standard (see e.g.\ \citep[Eq. 10.3.9]{BS08}) continuous trace inequality will be employed: There exists a constant $C>0$ such that 
\bel{ieq:cont-trace-ieq}
\norm{v}_{ \partial K } \leq C \left( h^{-1/2} \norm{v}_{ K } + h^{1/2} \norm{ \nabla v  }_{ K }   \right), \; \forall v \in [H^1(K)]^d. 
\ee
Before we proceed to work, let us give a remark discussing solutions with low regularity.
\begin{remark}
Below we will assume that the solution $u$ of \cref{eq:PDE-cont}-\cref{eq:u-measurements-omega} is in $ [H^{p+1}(\Omega)]^d$ for $p \geq 1$. 
If the Lam\'{e} coefficients are allowed to have jumps, then it is not realistic to assume that the solution enjoys such a high global regularity. 
However, according to \cref{ass:lame_reg_FEM} the jumps are limited to subdomains which are respected by the mesh. 
This would allow us to split $\Omega$ into subdomains $\Omega_i$ such that the restriction of $u$ is in $ [H^{p+1}(\Omega_i)]^d$ for each $i$ and treat 
each subdomain separately.  
To keep the analysis simple we will only consider such a scenario in our numerical experiments, see \cref{ssection:numexp_jump_shear}. 
\end{remark}
\begin{lem}[Weak consistency]\label{lem:weak-consistency}
Assume that $u \in [H^{p+1}(\Omega)]^d$ is a solution of \cref{eq:PDE-cont}-\cref{eq:u-measurements-omega}. 
Then there exists a constant $C>0$ such that there holds 
\bel{eq:weak-consistency}
\sum\limits_{j=1}^{p} \gammaCIP_{j}  J_j(\Pi_h u,\Pi_h u) \leq C h^{2p} \norm{u}_{ [H^{p+1}(\Omega)]^d }^2.  
\ee
\end{lem}
\begin{proof}
As discussed in \cref{sssection:stab_reg} we have $ \gammaCIP_{j}  J_j( u, u) =  \gammaCIP_{j}  J_j( u, \Pi_h u) = 0$ for $j=1,\ldots,p$. 
Therefore, inserting $u$, using the trace inequality of \cref{ieq:cont-trace-ieq} and approximation properties of $\Pi_h$ given in \cref{eq:approx-interp-ho}  yields: 
\begin{align*}
& \sum\limits_{j=1}^{p} \gammaCIP_{j} J_j(\Pi_h u, \Pi_h u) =  \sum\limits_{j=1}^{p} \gammaCIP_{j} \sum\limits_{F \in \mathcal{F}_i} \int\limits_{F} h^{2j-1} \jump{  (\nabla^{j-1} \sigma( \Pi_h u -u  )) \cdot \mathbf{n}  }^2  \mathrm{d}S \\\
& \leq C \sum\limits_{j=1}^{p}  h^{2j-2} \norm{  \nabla^{j-1} \sigma \left( \Pi_h u - u \right)   }_{ [L^2(\Omega)]^d }^2 + h^{2j} \norm{  \nabla^{j-1} \sigma \left( \Pi_h u - u \right)   }_{ [H^1(\Omega)]^d }^2  \\
& \leq C \sum\limits_{j=1}^{p}  h^{2j-2} \norm{  \left( \Pi_h u - u \right)   }_{ [H^{j}(\Omega)]^d }^2 + h^{2j} \norm{  \left( \Pi_h u - u \right)   }_{ [H^{j+1}(\Omega)]^d }^2  \\
& \leq C \sum\limits_{j=1}^{p}  h^{2j-2} h^{2(p+1-j)} \norm{ v }_{ [H^{p+1}(\Omega)]^d }^2 + h^{2j} h^{2(p+1-j-1)} \norm{ u }_{ [H^{p+1}(\Omega)]^d }^2   \\
& \leq C h^{2p} \norm{u}_{ [H^{p+1}(\Omega)]^d }^2.  
\end{align*}
\end{proof} 
\begin{cor}\label{cor:interp-stability-V_h}
Assume that $u \in [H^{p+1}(\Omega)]^d$ is a solution of \cref{eq:PDE-cont}-\cref{eq:u-measurements-omega}. 
Then there exists a constant $C>0$ such that  
\bel{ieq:interp-stability-V_h}
\norm{ \Pi_h u - u }_{V_h} \leq C h^{p} \norm{u}_{ [H^{p+1}(\Omega)]^d }.
\ee
\end{cor}
\begin{proof}
We have
\begin{align*}
\norm{ \Pi_h u - u }_{V_h}^2 &= s_{\gamma}(\Pi_h u -u, \Pi_h u - u ) + s_{\alpha}(\Pi_h u -u, \Pi_h u - u )    \\ 
	&= \sum\limits_{j=1}^{p}  \gammaCIP_{j}  J_j(\Pi_h u,\Pi_h u) + \gamma h^2  \norm{\mathcal{L}(\Pi_h u - u) }_{ \mathcal{T}_h  }^2 
	   + \alpha h^{2p} \norm{\Pi_h u -u}^2_{ \Omega  }. 
\end{align*}
In view of \cref{lem:weak-consistency}, it only remains to treat the last two terms.  We have 
\[
	\norm{\Pi_h u -u}^2_{ \Omega } \leq C h^{2(p+1)} \norm{u}_{ [H^{p+1}(\Omega)]^d  }^2
\] 
by \cref{eq:approx-interp-ho}. For the other term it follows from \cref{ass:lame_reg_FEM} on the coefficients 
and the approximation properties (\cref{eq:approx-interp-ho}) of $\Pi_h$ that 
\[
 h^2  \norm{\mathcal{L}(\Pi_h u - u) }_{ \mathcal{T}_h }^2 
  \leq  C h^2  \norm{\Pi_h u - u }_{ [H^2(\mathcal{T}_h)]^d }^2   
  \leq C h^2 h^{2(p-1)} \norm{u}_{ [H^{p+1}(\Omega)]^d  }^2.
\]
Combining these estimates yields the claim.
\end{proof}

\section{Error analysis}\label{section:ErrorAnalysis}
This section is concerned with the derivation of covergence rates for the stabilized finite element method introduced in \cref{section:Discretisation}. 
In \cref{ssection:unperturbed-data} we first consider the case of unperturbed data. 
The perturbed case can then be treated in \cref{ssection:perturbed-data} by minor modification of the proofs for the unperturbed situation. 

\subsection{Unperturbed data}\label{ssection:unperturbed-data}
To obtain error estimates, we will apply the conditional stability estimate from \cref{cor:cond-stability} to the error $u-u_h$. 
Controlling the arising terms on the right hand side of \cref{eq:cond-stability} requires estimates on the residual 
\[
\langle r,w \rangle := a_h(u_h-u,w) = a_h(u_h,w) - (f,w)_{ \Omega }, \quad w \in V_0. 
\]
The next lemma provides one of the essential bounds for this purpose.

\begin{lem}\label{lem:a_h-control-V_h}
\begin{enumerate}[label=(\alph*)]
\item 
There exists a constant $C>0$ such that 
\bel{ieq:a_h-control-V_h}
a_h(u,v) \leq C \norm{u}_{V_h} \left( h^{-1} \norm{v}_{ \Omega } + \norm{ \nabla v}_{ \Omega } \right), 
\ee
for all $ u \in V_h + [H^{p+1}(\Omega)]^{d}$ and $ v \in V_0$.
\item There exists a constant $C>0$ such that 
\bel{ieq:s_beta-control-V_h}
s_{\beta}(u, w_h ) \leq C  \norm{u}_{V_h} \norm{ \nabla w_h}_{ \Omega } , \; \forall u \in V_h + [H^{p+1}(\Omega)]^{d}, \; \forall  w_h \in W_h.  
\ee
\end{enumerate}
\end{lem}
\begin{proof}
\begin{enumerate}[label=(\alph*)]
\item 
Element-wise integration by parts yields 
\begin{align*}
a_h(u,v) &= \sum\limits_{K \in \mathcal{T}_h} \int\limits_{K} \left[ \sigma(u):\nabla v - \rho u v  \right] \; \mathrm{d}x  \\
	   &= \sum\limits_{K \in \mathcal{T}_h} \int\limits_{K} \left[  - \nabla \cdot \sigma(u) -   \rho u \right] v \; \mathrm{d}x 
	    + \sum\limits_{K \in \mathcal{T}_h} \int\limits_{\partial K}   \sigma(u)  \cdot  \mathbf{n}   v  \;   \mathrm{d}S \\
  &= \sum\limits_{K \in \mathcal{T}_h} \int\limits_{K} \mathcal{L} u v \; \mathrm{d}x 
	+ \sum\limits_{F \in \mathcal{F}_i} \int\limits_{F} \jump{  \sigma(u)  \cdot  \mathbf{n}  } v \; \mathrm{d}S \\ 
	&:= \mathrm{I} + \mathrm{II}.
\end{align*}
We control the first term by means of the Galerkin least squares stabilization:
	\[ \mathrm{I} \leq \left( h^2 (\mathcal{L} u, \mathcal{L} u)_{ \mathcal{T}_h } \right)^{1/2}  h^{-1} \norm{v}_{ \Omega  }  
	      \leq C \norm{u}_{V_h} h^{-1} \norm{v}_{ \Omega }. 
	\]
 The penalty on the normal jumps of $\sigma(u)$ over the facets allows to estimate the second term: 
\begin{align*}
	\mathrm{II} &\leq C \left(  \sum\limits_{F \in \mathcal{F}_i}  \int\limits_{F}  h \jump{ \sigma(u) \cdot \mathbf{n} }^2  \; \mathrm{d}S   \right)^{1/2} 
	\left( \sum\limits_{F \in \mathcal{F}_i} h^{-1}  \norm{v}_{ F  }^2  \right)^{1/2} \\  
	& \leq C \left(  \sum\limits_{F \in \mathcal{F}_i}  \int\limits_{F}  h \jump{ \sigma(u) \cdot \mathbf{n} }^2  \; \mathrm{d}S   \right)^{1/2}  
	 \left( \sum\limits_{K \in \mathcal{T}_h} h^{-2}  \norm{v}_{ K }^2  + \norm{\nabla v}_{ K }^2    \right)^{1/2} \\
	& \leq C J_1(u,u)^{1/2} \left( h^{-1} \norm{v}_{ \Omega } + \norm{ \nabla v}_{  \Omega } \right),
\end{align*}
where the trace inequality in \cref{ieq:cont-trace-ieq} has been employed. 
Note that $J_1(u,u)^{1/2} \leq C  \norm{u}_{V_h}$ thanks to $\gammaCIP_1 > 0$. 
Combining both contributions yields the claim.
\\ 
\item Making use of the inverse inequalities 
$\norm{w_h}_F \leq C h^{-1/2} \norm{w_h}_K$ and $\norm{\nabla w_h}_K \leq C h^{-1} \norm{w_h}_K$ yields
\begin{align*}
	\sum\limits_{j=1}^{p} \abs{\beta_j} J_j(w_h,w_h)   &= \sum\limits_{j=1}^{p} \abs{\beta_j} \sum\limits_{F \in \mathcal{F}_i} \int\limits_{F} h^{2j-1} \jump{  (\nabla^{j-1} \sigma(w_h)) \cdot \mathbf{n}  }^2   \; \mathrm{d}S  \\
	&\leq C  \sum\limits_{j=1}^{p} \sum\limits_{K \in \mathcal{T}_h}  h^{2j-1} h^{-1}  \norm{ \nabla^{j-1} \sigma(w_h) }_{K}^2 \\ 
	&\leq C \sum\limits_{j=1}^{p} \sum\limits_{K \in \mathcal{T}_h}  h^{2j-1} h^{-1} h^{-2(j-1)}  \norm{ \sigma(w_h) }_{K}^2 \\ 
	& \leq C  \sum\limits_{j=1}^{p} \sum\limits_{K \in \mathcal{T}_h}  \norm{ \nabla w_h }_{K}^2.  
\end{align*}
Combining this with the Cauchy-Schwarz inequality 
\[
s_{\beta}(u, w_h ) \leq  \left( \sum\limits_{j=1}^{p} \abs{\beta_j} J_j(u,u) \right)^{1/2} \left( \sum\limits_{j=1}^{p} \abs{\beta_j} J_j(w_h,w_h) \right)^{1/2}   \leq C  \norm{u}_{V_h} \norm{ \nabla w_h}
\]
yields the claim. Here we used that 
\[
 \sum\limits_{j=1}^{p} \abs{\beta_j} J_j(u,u) \leq C \sum\limits_{j=1}^{p}  \gammaCIP_{j} J_j(u,u),  
\]
which follows from the assumption given in \cref{eq:penalty-param-assumption} on the penalty parameters.
\end{enumerate}
\end{proof}
We will later apply \cref{lem:a_h-control-V_h} (a) to $u_h -u$, which will result in a term $\norm{u_h-u}_{V_h}$ on the right hand side of \cref{ieq:a_h-control-V_h}. 
Since
\[  \norm{u_h-u}_{V_h}  \leq  \norm{u_h- \Pi_h u}_{V_h} + \norm{ \Pi_h u - u }_{V_h}   \]
and we can already control $\norm{ \Pi_h u - u }_{V_h}$ by \cref{ieq:interp-stability-V_h}, it remains to 
consider $\norm{u_h- \Pi_h u}_{V_h}$. 
To this end, we prove the next lemma.
\begin{lem}\label{lem:ieq:discr-error-s-norm}
Assume that $u \in [H^{p+1}(\Omega)]^d$ is a solution of \cref{eq:PDE-cont}-\cref{eq:u-measurements-omega} and let $(u_h,z_h) \in V_h \times W_h$ be the solution to \cref{eq:opt-compact}. 
Then there exists $C>0$ such that for all $h \in (0,1)$ it holds that 
\bel{ieq:discr-error-s-norm}
\norm{(u_h - \Pi_h u, z_h)}_{s} \leq C h^{p} \norm{u}_{ [H^{p+1}(\Omega)]^d  }.
\ee
\end{lem}
\begin{proof} 
It suffices to prove that for $(v_h,w_h) \in V_h \times W_h$ the inequality
\bel{ieq:discr-error-s-norm-aux}
A[ (u_h - \Pi_h u,z_h),(v_h,w_h) ] \leq C h^{p} \norm{u}_{ [H^{p+1}(\Omega)]^d  } \norm{ (v_h,w_h) }_{s} 
\ee
holds, because then the inf-sup condition in \cref{eq:inf-sup} yields 
\[
C \norm{(u_h - \Pi_h u, z_h)}_{s} \leq  \!\!\!\! 
 \sup_{ (v_h,w_h) \in V_h \times W_h} \!\!\!\! \frac{ A[(u_h -\Pi_h u,z_h),(v_h,w_h)]   }{ \norm{ (v_h,w_h)}_s   } \leq C h^{p} \norm{u}_{ [H^{p+1}(\Omega)]^d  }.
\]
To prove \cref{ieq:discr-error-s-norm-aux}, we use \cref{eq:opt-compact} to arrive at
\begin{align*}
& A[ (u_h - \Pi_h u,z_h),(v_h,w_h) ] =  A[ (u_h,z_h),(v_h,w_h) ] \\
& \quad - (\Pi_h u,v_h)_{ \omega }  - s_{\gamma}(\Pi_h u, v_h) - s_{\alpha}(\Pi_h u, v_h)  - a_h(\Pi_h u,w_h) -  s_{\beta}(\Pi_h u , w_h )  \\ 
&=  (u_{\omega},v_h)_{ \omega } + s_{\gamma}(u,v_h) +  \underbrace{(f,w_h)_{ \Omega }}_{= a_h(u,w_h)  } \\ 
	  & \quad  - (\Pi_h u,v_h)_{ \omega } - s_{\gamma}(\Pi_h u, v_h) - s_{\alpha}(\Pi_h u, v_h) - a_h(\Pi_h u,w_h) -  s_{\beta}(\Pi_h u , w_h ) \\
	&=  (u - \Pi_h u,v_h)_{ \omega } + a_h(u - \Pi_h u,w_h) + s_{\gamma}(u - \Pi_h u, v_h) - s_{\alpha}(\Pi_h u, v_h) +  s_{\beta}(u -\Pi_h u , w_h ). 
\end{align*}
Here we also employed the consisteny of the jump penalties, c.f.\ \cref{eq:CIP_consistency}.
\begin{itemize}
\item The first term is bounded by using Cauchy-Schwarz and the approximation properties of $\Pi_h$:  
\[  (u - \Pi_h u,v_h)_{ \omega } \leq \norm{u - \Pi_h u}_{ \Omega } \norm{v_h }_{ \omega } \leq C h^{p+1} \norm{u}_{ [H^{p+1}(\Omega)]^d} \norm{v_h }_{\omega }. \]
\item For the second term we have 
\begin{align*}
     a_h(u-\Pi_h u,w_h) &= \int\limits_{\Omega} \left[ \sigma(u-\Pi_h u)  :\nabla w_h  - \rho (u -\Pi_h u) w_h  \right] \mathrm{d}x \\ 
   & \leq C \left( \norm{ \nabla ( u_h - \Pi_h u) }_{ \Omega } \norm{ \nabla w_h  }_{  \Omega } 
		  + \norm{ u_h - \Pi_h u}_{ \Omega }  \norm{   w_h }_{ \Omega } \right) \\
   & \leq C \left(  h^{p} \norm{u}_{ [H^{p+1}(\Omega)]^d}    \norm{ \nabla w_h  }_{ \Omega } 
		  +  h^{p+1} \norm{u}_{ [H^{p+1}(\Omega)]^d}    \norm{   w_h }_{ \Omega } \right) \\
   & \leq C h^{p} \norm{u}_{ [H^{p+1}(\Omega)]^d} \norm{ w_h}_{W_h},
\end{align*}
where we used the approximation properties of $\Pi_h$ and Friedrichs inequality 
\bel{ieq:Friedrich}
	\norm{  w_h }_{ [L^2(\Omega)]^d } \leq  C  \norm{ \nabla  w_h }_{ [L^2(\Omega)]^d }, \quad \forall w_h \in W_h \subset V_0. 
\ee 
\item For the third term we obtain from \cref{ieq:interp-stability-V_h} that 
	\[ s_{\gamma}(u - \Pi_h u, v_h) \leq \norm{u - \Pi_h u}_{V_h} \norm{ v_h }_{V_h} \leq C h^{p} \norm{u}_{ [H^{p+1}(\Omega)]^d} \norm{v_h}_{V_h}. \] 
\item The second to last term is bounded by
\begin{align*}
	s_{\alpha}(\Pi_h u, v_h) & = \sqrt{\alpha} h^{p}(\Pi_h u, \sqrt{\alpha} h^{p}v_h)_{ \Omega } 
				   \leq \sqrt{\alpha} h^p \norm{u}_{ [H^1(\Omega)]^d } \norm{ v_h }_{V_h}. 
\end{align*} 
\item For the last term we can use \cref{ieq:s_beta-control-V_h} to obtain
\begin{equation*}
s_{\beta}(u -\Pi_h u , w_h ) \leq C  \norm{u - \Pi_h u}_{V_h} \norm{\nabla w_h}_{ \Omega } \leq C h^{p} \norm{u}_{ [H^{p+1}(\Omega)]^d} \norm{\nabla w_h}_{ \Omega  } .
\end{equation*}
\end{itemize}
Combining these estimates yields \cref{ieq:discr-error-s-norm-aux}.
\end{proof}
We are now in a position to derive an $L^2$-error estimate for unperturbed data. 
For $p=1$ it is comparable with \citep[Theorem 1]{BNO19} for the Helmholtz equation except that the dependence on the wavenumber is implicit in our estimate.
\begin{theorem}\label{thm:L2-error-unperturbed}
Let the subdomains $\omega$ and $B$ of $\Omega$ be defined as in \cref{cor:cond-stability}. 
Assume that $u \in [H^{p+1}(\Omega)]^d$ is a solution to \cref{eq:PDE-cont}-\cref{eq:u-measurements-omega} and let $(u_h,z_h) \in V_h \times W_h$ be the solution to \cref{eq:opt-compact}. 
Then there exists $C>0$ and $\tau \in (0,1)$ such that  
\bel{ieq:L2-error-estimate-unperturbed}
\norm{u - u_h}_{ B } \leq C h^{\tau p} \norm{u}_{ [H^{p+1}(\Omega)]^d  }.
\ee
\end{theorem}
\begin{proof}
Consider the residual 
\[ \langle r,w \rangle := a_h(u_h-u,w) = a_h(u_h,w) - (f,w)_{ \Omega }, \quad w \in V_0. \]
Taking $v_h = 0$ in \cref{eq:opt-compact} yields: 
\[ a_h(u_h,w_h)    = (f,w_h)_{\Omega} + s^{\ast}(z_h,w_h) - s_{\beta}(u_h,w_h)  \quad  \forall w_h \in W_h.     \]
Using this identity with $w_h = \Pi_h w$ implies
\begin{align*}
\langle r,w \rangle &= a_h(u_h,w) - (f,w)_{ \Omega } - a_h(u_h,\Pi_h w) + a_h(u_h,\Pi_h w)  \\ 
	 & =  a_h(u_h,w - \Pi_h w) - (f, w - \Pi_h w )_{ \Omega } + s^{\ast}(z_h,\Pi_h w ) - s_{\beta}(u_h, \Pi_h w )   \\ 
	 & =  a_h(u_h - u,w - \Pi_h w) + s^{\ast}(z_h,\Pi_h w ) - s_{\beta}(u_h - u, \Pi_h w ). 
\end{align*}
\begin{itemize}
\item From \cref{lem:a_h-control-V_h} we obtain that 
\begin{align*}
a_h(u_h - u,w - \Pi_h w) & \leq C \norm{u_h -u}_{V_h} \left( h^{-1} \norm{ w - \Pi_h w }_{ \Omega } + \norm{ \nabla \left( w - \Pi_h w \right)   }_{ \Omega } \right)  \\ 
	& \leq C \norm{u_h -u}_{V_h} \norm{w}_{ [H^1(\Omega)]^d},
\end{align*}
by the properties in \cref{eq:H1-stability-SZ} and \cref{eq:approx-interp-ho} of $\Pi_h$.
Further, from \cref{lem:ieq:discr-error-s-norm} and \cref{ieq:interp-stability-V_h} we obtain 
\bel{eq:triangle_ieq_Vh}
\norm{u_h -u}_{V_h} \leq \norm{u_h - \Pi_h u}_{V_h} + \norm{\Pi_h u -u}_{V_h} \leq 
C h^{p} \norm{u}_{ [H^{p+1}(\Omega)]^d  }.
\ee 
\item To bound the second term, we again use \cref{lem:ieq:discr-error-s-norm} and the $H^1$-stability of $\Pi_h$: 
\[  s^{\ast}(z_h,\Pi_h w ) \leq \norm{z_h}_{W_h} \norm{ \Pi_h w }_{W_h} \leq C  h^{p} \norm{u}_{ [H^{p+1}(\Omega)]^d  } \norm{w}_{ [H^{1}(\Omega)]^d }. \]
\item 
The last term is treated by invoking \cref{ieq:s_beta-control-V_h} and then proceeding as in \cref{eq:triangle_ieq_Vh}:
\begin{align*}
s_{\beta}(u_h - u, \Pi_h w ) \leq C  \norm{u-u_h}_{V_h} \norm{w}_{ [H^1(\Omega)]^d} \leq C  h^{p} \norm{u}_{ [H^{p+1}(\Omega)]^d  } \norm{w}_{ [H^{1}(\Omega)]^d }. 
\end{align*}
\end{itemize}
Hence, the following residual norm estimate holds
\[
\norm{r}_{ V_{0}^{\prime} } \leq C  h^{p} \norm{u}_{ [H^{p+1}(\Omega)]^d  }. 
\]
Using the conditional stability estimate from \cref{cor:cond-stability} for $u-u_h$ (note that in \cref{eq:cond-stability} we have $f =  r$ in $V_0^{\prime}$) yields the following error estimate
\[
\norm{u-u_h}_{ B } \leq C \left( h^{p} \norm{u}_{ [H^{p+1}(\Omega)]^d  }  + \norm{u-u_h}_{ \omega }  \right)^{\tau} 
	  \left(  h^{p} \norm{u}_{ [H^{p+1}(\Omega)]^d  }  + \norm{u-u_h}_{ \Omega }     \right)^{1-\tau}.
\]
\begin{itemize}
\item From \cref{eq:approx-interp-ho} and \cref{lem:ieq:discr-error-s-norm} we obtain 
\[  \norm{u-u_h}_{ \omega  }  \leq  \norm{u- \Pi_h u}_{ \omega }  + \norm{\Pi_h u - u_h}_{ \omega }  \leq C h^{p} \norm{u}_{ [H^{p+1}(\Omega)]^d  }. \]
\item We also have 
\[ \norm{u-u_h}_{ \Omega } \leq  \norm{u- \Pi_h u}_{ \Omega } + \norm{u_h - \Pi_h u}_{ \Omega } \leq C h^{p+1} \norm{u}_{ [H^{p+1}(\Omega)]^d  } + \norm{u_h - \Pi_h u}_{ \Omega }.   \]
It remains to estimate $\norm{u_h - \Pi_h u}_{ [L^2(\Omega)]^d }$.
By definition of $s_{\alpha}(\cdot,\cdot)$, see \cref{eq:s_alpha-def},  we have 
\begin{align*}
\norm{u_h - \Pi_h u}_{ \Omega } &= \alpha^{-1/2} h^{-p} s_{\alpha}( u_h - \Pi_h u, u_h - \Pi_h u )^{1/2} \\ 
                & \leq C h^{-p} \norm{u_h - \Pi_h u}_{V_h}  \leq C  \norm{u}_{ [H^{p+1}(\Omega)]^d  },  
\end{align*}
where the last inequality follows by \cref{lem:ieq:discr-error-s-norm}.  
\end{itemize}
It follows that 
\[ \norm{u-u_h}_{ B } \leq C \left(h^p \norm{u}_{ [H^{p+1}(\Omega)]^d  } \right)^{\tau} \left( \norm{u}_{ [H^{p+1}(\Omega)]^d  } \right)^{1-\tau}  
 = C h^{p \tau} \norm{u}_{ [H^{p+1}(\Omega)]^d  }.     \]
\end{proof}

\subsection{Perturbed data}\label{ssection:perturbed-data}

We now proceed to the case of perturbed data 
\[  \tilde{u}_{\omega} := u_{\omega} + \delta u, \quad  
    \tilde{f} := f + \delta f 
\]
with unperturbed data $u_{\omega},f$ in \cref{eq:PDE-cont}, respectively \cref{eq:u-measurements-omega} and perturbations 
$\delta u \in [L^2(\omega)]^d$ and $\delta f \in [L^2(\Omega)]^d$ measured by 
\bel{eq:perturb-norm}
\delta(\tilde{u}_{\omega},\tilde{f}) := \norm{\delta u}_{ \omega  }  + h \norm{ \delta f }_{ \Omega  } + \norm{ \delta f }_{H^{-1}(\Omega)  }. 
\ee
In view of \cref{eq:s_gamma_u_vh_known}, we have
\[ 
\gammaGLS h^2 (\tilde{f},\mathcal{L} v_h)_{ \mathcal{T}_h } = s_{\gamma}(u,v_h) + \gammaGLS h^2 (\delta f,\mathcal{L} v_h)_{ \mathcal{T}_h },
\]
so that the saddle points of the corresponding perturbed Lagrangian now satisfy: 
\bel{eq:opt-compact-perturb}
A[(u_h,z_h),(v_h,w_h)] = (\tilde{u}_{\omega},v_h)_{ \omega } + s_{\gamma}(u,v_h) + \gammaGLS h^2 (\delta f,\mathcal{L} v_h)_{ \mathcal{T}_h } + (\tilde{f},w_h)_{ \Omega }  
\ee
for all $(v_h,w_h) \in V_h \times W_h$.  
Let us first prove the analogue of \cref{lem:ieq:discr-error-s-norm} for perturbed data. 
\begin{lem}\label{lem:ieq:discr-error-s-norm-perturb}
Assume that $u \in [H^{p+1}(\Omega)]^d$ is a solution of the unperturbed problem in \cref{eq:PDE-cont}-\cref{eq:u-measurements-omega} and let $(u_h,z_h) \in V_h \times W_h$ be the solution of the
perturbed problem in \cref{eq:opt-compact-perturb}. 
Then there exists $C>0$ such that for all $h \in (0,1)$ it holds that 
\bel{ieq:discr-error-s-norm-perturb}
	\norm{(u_h - \Pi_h u, z_h)}_{s} \leq C \left( h^{p} \norm{u}_{ [H^{p+1}(\Omega)]^d } + \delta(\tilde{u}_{\omega},\tilde{f}) \right). 
\ee
\end{lem}
\begin{proof}
Proceeding as in the proof of \cref{lem:ieq:discr-error-s-norm} we use \cref{eq:opt-compact-perturb} to arrive at  
\begin{align*}
& A[ (u_h - \Pi_h u,z_h),(v_h,w_h) ]  \\
& =  (u_{\omega},v_h)_{ \omega } + (\delta u,v_h)_{ \omega } 
   + s_{\gamma}(u,v_h) + \gammaGLS h^2 (\delta f,\mathcal{L} v_h)_{ \mathcal{T}_h },
   +  a_h(u,w_h)  + (\delta f, w_h)_{ \Omega }  \\
	  & \quad  - (\Pi_h u,v_h)_{ \omega } - s_{\gamma}(\Pi_h u, v_h) - s_{\alpha}(\Pi_h u, v_h) - a_h(\Pi_h u,w_h) -  s_{\beta}(\Pi_h u , w_h )  \\
& =  (u - \Pi_h u,v_h)_{ \omega } + a_h(u - \Pi_h u,w_h) + s_{\gamma}(u - \Pi_h u, v_h) - s_{\alpha}(\Pi_h u, v_h) +  s_{\beta}(u -\Pi_h u , w_h ) \\ 
    & \quad + \gammaGLS h^2 (\delta f,\mathcal{L} v_h)_{ \mathcal{T}_h } + (\delta u,v_h)_{ \omega } +  (\delta f, w_h)_{ \Omega }. 
\end{align*}
The terms in the second to last line are bounded as in the proof of \cref{lem:ieq:discr-error-s-norm}.
The terms including the perturbations are estimated by 
\begin{align*} 
   & \gammaGLS h^2 (\delta f,\mathcal{L} v_h)_{ \mathcal{T}_h  } + (\delta u,v_h)_{ \omega } +  (\delta f, w_h)_{ \Omega }  \\  
   & \leq  \gammaGLS h \norm{ \delta f }_{ \Omega }  h \norm{\mathcal{L} v_h}_{ \mathcal{T}_h } + \norm{\delta u}_{ \omega  } \norm{v_h}_{ \omega  } 
	  + \norm{ \delta f }_{ H^{-1}(\Omega) } \norm{w_h}_{ H^1(\Omega) }   \\   
   & \leq  C  \left( h \norm{ \delta f }_{ \Omega } + \norm{ \delta f }_{ H^{-1}(\Omega) } +\norm{\delta u}_{ \omega  }  \right) \left(  \norm{ v_h}_{V_h } + \norm{w_h}_{ W_h }  \right)   \\
	& \leq C \delta(\tilde{u}_{\omega},\tilde{f}) \norm{(v_h,w_h) }_{s}, 
\end{align*}
where Friedrichs inequality, see \cref{ieq:Friedrich}, has been employed.
\end{proof}

With this lemma being established we can show the analogue of \cref{thm:L2-error-unperturbed} for perturbed data. 
Our result is comparable with \citep[Theorem 3]{BNO19} for the Helmholtz equation obtained for $p=1$.
It is also comparable with \citep[Theorem 5.7]{BDE21} for the case of higher polynomial orders $p$ except that the latter publication even controls the $H^1$-norm 
in $B$.
This is out of scope here since it would require a conditional stability estimate which additionally controls the gradient of $u$, see \citep[Lemma 3.1]{BDE21}.
\begin{theorem}\label{thm:L2-error-perturbed-data}
Let the subdomains $\omega$ and $B$ of $\Omega$ be defined as in \cref{cor:cond-stability}. 
Assume that $u \in [H^{p+1}(\Omega)]^d$ is a solution to the unperturbed problem in \cref{eq:PDE-cont}-\cref{eq:u-measurements-omega} and let $(u_h,z_h) \in V_h \times W_h$ be the solution 
of the perturbed problem in \cref{eq:opt-compact-perturb}. 
Then there exists $C>0$ and $\tau \in (0,1)$ such that  
\bel{ieq:L2-error-estimate-perturbed}
\norm{u - u_h}_{B} \leq C h^{\tau p} \left(  \norm{u}_{ [H^{p+1}(\Omega)]^d } + h^{-p} \delta(\tilde{u}_{\omega},\tilde{f}) \right). 
\ee
\end{theorem}

\begin{proof}
Following the proof of \cref{thm:L2-error-unperturbed}, the residual can now be written as 
\[
\langle r,w \rangle =  
   a_h(u_h-u,w - \Pi_h w) + s^{\ast}(z_h,\Pi_h w ) - s_{\beta}(u_h - u, \Pi_h w ) + (\delta f,\Pi_h w )_{ \Omega }. 
\]
We estimate the first three terms similar as in the proof of \cref{thm:L2-error-unperturbed} but now appealing to \cref{lem:ieq:discr-error-s-norm-perturb}  
instead of \cref{lem:ieq:discr-error-s-norm}. 
\begin{itemize}
\item As in the proof of \cref{thm:L2-error-unperturbed} we obtain  
\begin{align*} 
a_h(u_h-u,w - \Pi_h w) & \leq C  \norm{u_h -u}_{V_h} \norm{w}_{ [H^1(\Omega)]^d} \\ 
          & \leq C \left(   \norm{\Pi_h u - u}_{V_h} + \norm{ u_h - \Pi_h u  }_{V_h}   \right) \norm{w}_{ [H^1(\Omega)]^d}  \\ 
          & \leq C \left(  h^{p} \norm{u}_{ [H^{p+1}(\Omega)]^d } + \delta(\tilde{u}_{\omega},\tilde{f})  \right) \norm{w}_{ [H^1(\Omega)]^d}, 
\end{align*}
where \cref{lem:ieq:discr-error-s-norm-perturb} and \cref{ieq:interp-stability-V_h} have been emloyed. 
\item Furthermore,  
	\[  s^{\ast}(z_h,\Pi_h w ) \leq \norm{z_h}_{W_h} \norm{ \Pi_h w }_{W_h} \leq C 
		\left(  h^{p} \norm{u}_{ [H^{p+1}(\Omega)]^d  } + \delta(\tilde{u}_{\omega},\tilde{f}) \right) \norm{w}_{ [H^{1}(\Omega)]^d )}, \]
by invoking \cref{lem:ieq:discr-error-s-norm-perturb} again. 
\item 
The term involving $s_{\beta}$ is treated by using the inequality in \cref{ieq:s_beta-control-V_h} and then proceeding as above to estimate $\norm{u -u_h}_{V_h}$, i.e.   
\begin{align*}
s_{\beta}(u_h - u, \Pi_h w ) &\leq C  \norm{u-u_h}_{V_h} \norm{w}_{ [H^1(\Omega)]^d} \\
	 &\leq C \left(  h^{p} \norm{u}_{ [H^{p+1}(\Omega)]^d } + \delta(\tilde{u}_{\omega},\tilde{f})  \right) \norm{w}_{ [H^1(\Omega)]^d}.
\end{align*}
\item The perturbation term is easily bounded by using Cauchy-Schwarz and the stability of the interpolation: 
\[  
(\delta f,\Pi_h w )_{ \Omega } \leq  \norm{ \delta f }_{ H^{-1}( \Omega) } \norm{  \Pi_h w }_{ H^1(\Omega) } \leq C \delta(\tilde{u}_{\omega},\tilde{f}) \norm{ w }_{ H^1(\Omega) }. 
\] 
\end{itemize}
It follows that
\[ 
\norm{r}_{V_0^{\prime}} \leq C \left(  h^{p} \norm{u}_{ [H^{p+1}(\Omega)]^d  } + \delta(\tilde{u}_{\omega},\tilde{f})   \right).
\]
The conditional stability estimate from \cref{cor:cond-stability} therefore leads to the error estimate
\begin{align*}
	\norm{u-u_h}_{ B } \leq &C \left(  h^{p} \norm{u}_{ [H^{p+1}(\Omega)]^d  }  + \delta(\tilde{u}_{\omega},\tilde{f}) + \norm{u-u_h}_{ \omega }  \right)^{\tau} \\ 
	  & \times \left( h^{p} \norm{u}_{ [H^{p+1}(\Omega)]^d  } + \delta(\tilde{u}_{\omega},\tilde{f}) + \norm{u-u_h}_{ \Omega } \right)^{1-\tau}.
\end{align*}
\begin{itemize}
\item From \cref{lem:ieq:discr-error-s-norm-perturb} and equation \cref{eq:approx-interp-ho} we obtain 
\begin{align*}
\norm{u-u_h}_{ \omega }  &\leq  \norm{u- \Pi_h u}_{ \omega }  + \norm{\Pi_h u - u_h}_{ \omega }  \\ 
                                  & \leq C \left(  h^{p} \norm{u}_{ [H^{p+1}(\Omega)]^d  }  + \delta(\tilde{u}_{\omega},\tilde{f}) \right). 
\end{align*}
\item We also have
\begin{align*}
\norm{u-u_h}_{ \Omega } & \leq  \norm{u- \Pi_h u}_{ \Omega } + \norm{u_h - \Pi_h u}_{ \Omega } \\ 
	    & \leq C h^{p+1} \norm{u}_{ [H^{p+1}(\Omega)]^d  } + \norm{u_h - \Pi_h u}_{ \Omega }  
\end{align*}
and by definition of $s_{\alpha}(\cdot,\cdot)$, see \cref{eq:s_alpha-def}, and \cref{lem:ieq:discr-error-s-norm-perturb} it holds that 
\begin{align*}
\norm{u_h - \Pi_h u}_{ \Omega }  & \leq C h^{-p} \norm{u_h - \Pi_h u}_{V_h}  \\
					  & \leq C \left( \norm{u}_{ [H^{p+1}(\Omega)]^d  }  + h^{-p} \delta(\tilde{u}_{\omega},\tilde{f})  \right). 
\end{align*} 
\end{itemize}
It follows that 
\begin{align*}
	\norm{u-u_h}_{ B }  &\leq C \left( h^p \left[  \norm{u}_{ [H^{p+1}(\Omega)]^d  }  + h^{-p} \delta(\tilde{u}_{\omega},\tilde{f}) \right] \right)^{\tau} 
	 \left(  \norm{u}_{ [H^{p+1}(\Omega)]^d  }  + h^{-p} \delta(\tilde{u}_{\omega},\tilde{f})   \right)^{1-\tau}  \\
         &= C h^{\tau p} \left( \norm{u}_{ [H^{p+1}(\Omega)]^d  }  + h^{-p} \delta(\tilde{u}_{\omega},\tilde{f}) \right).
\end{align*}
\end{proof}

\begin{remark}
Notice that if $h < h_{\mathrm{min}}$ for $h_{\mathrm{min}} := (\delta(\tilde{u}_{\omega},\tilde{f}) / \norm{u}_{ [H^{p+1}(\Omega)]^d })^{1/p} $ the data perturbation term in \cref{ieq:L2-error-estimate-perturbed} dominates so that further refinement of the mesh 
will lead to poorer accuracy. 
Hence, refinement should be stopped at $h =h_{\mathrm{min}}$. 
Alternatively, the coefficient in front of the Tikhonov term in \cref{eq:s_alpha-def} can be made lower bounded to ensure that stagnation of the error occurs for $h < h_{\mathrm{min}}$ 
as explained in detail in \citep[Remark 5.1]{BDE21}.
However, this also requires an estimate of $\norm{u}_{ [H^{p+1}(\Omega)]^d }$ and the noise level. 
\end{remark}

\section{Numerical experiments}\label{section:numexp}

In this section we present a selection of numerical experiments to confirm the analytically derived error estimate of \cref{thm:L2-error-perturbed-data} and shed light on
several additional aspects which exceed the scope of our present analysis. 
All experiments have been implemented using the open-source computing platform \texttt{FEniCSx} \citep{UFL14,BasixJoss22}. 
To check our results we also implemented some of the numerical experiments in \texttt{Netgen/ NGSolve} \citep{JS97,JS14} and observed 
a qualitatively good agreement. 
A docker image containing all software and instructions to reproduce the numerical experiments shown in this paper can be obtained from 
the \texttt{zenodo} repository: \cite{BP22_zenodo}. \par 
In all the numerical experiments we will set $\rho = -k^2$ for a positive constant $k>0$ representing the wavenumber.
For the experiments in \cref{ssection:numexp_stab_tuning}-\cref{ssection:numexp_pollution} we consider the following geometrical setup.
Let $\Omega = [0,1]^2$ be the unit square and the measurement $\omega$ and target domain $B$ be given by
\bel{eq:subdomains-convex}
\omega = \Omega \setminus [0.1,0.9] \times [0.25,1] \text{ and }  
B = \Omega \setminus [0.1,0.9] \times [0.95,1].
\ee
These subdomains are displayed in \cref{fig:subdomains-convex}. 
We consider a sequence of meshes which are obtained by successive refinements of an initial mesh which is shown in \cref{fig:omega-Ind-convex-level0}.
All these meshes fulfill our assumption (see \cref{ssection:FEspaces}) of being fitted to the subdomains.
\par 

\begin{figure}[htbp]
\centering
\subfloat[ $\omega$ on coarsest mesh.]%
{\label{fig:omega-Ind-convex-level0}
\includegraphics[width=.4\textwidth]{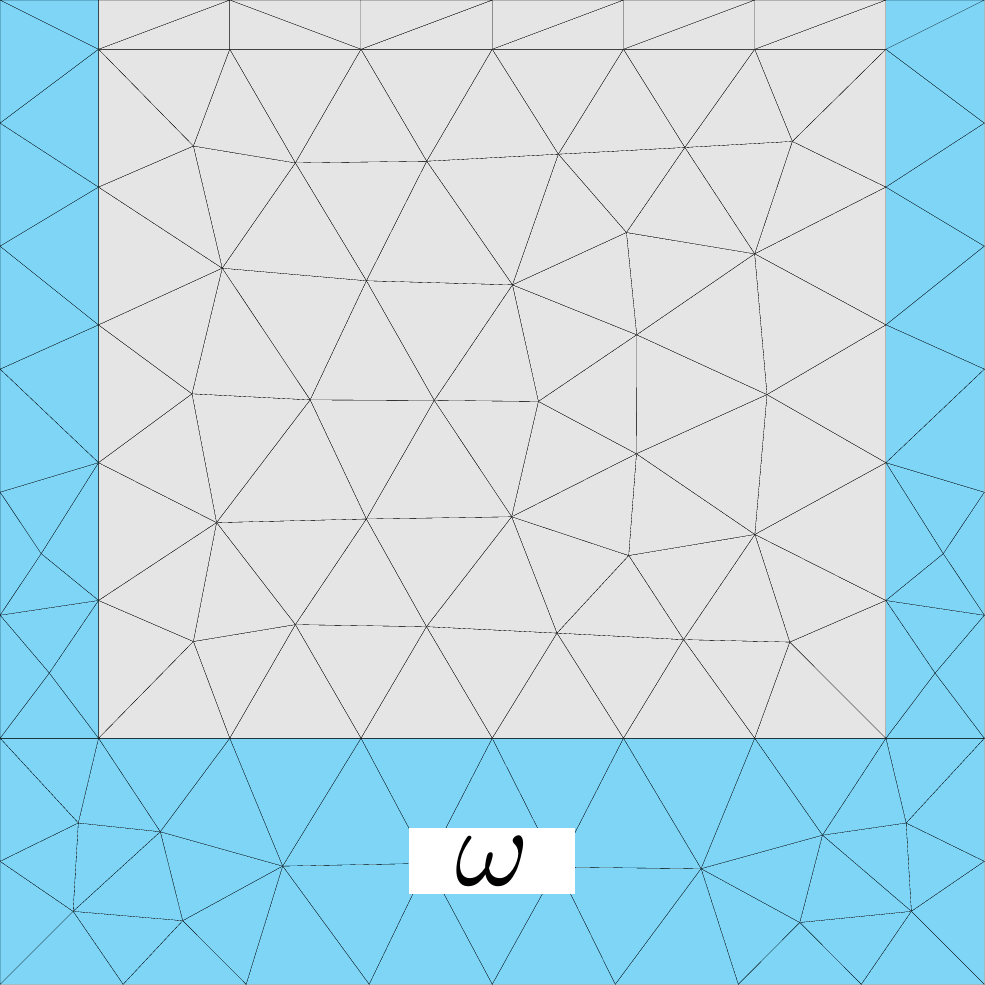}}
\centering
\subfloat[ $B$ on refined mesh.]
{ \label{fig:B-Ind-convex-level1} %
\includegraphics[width=.4\textwidth]{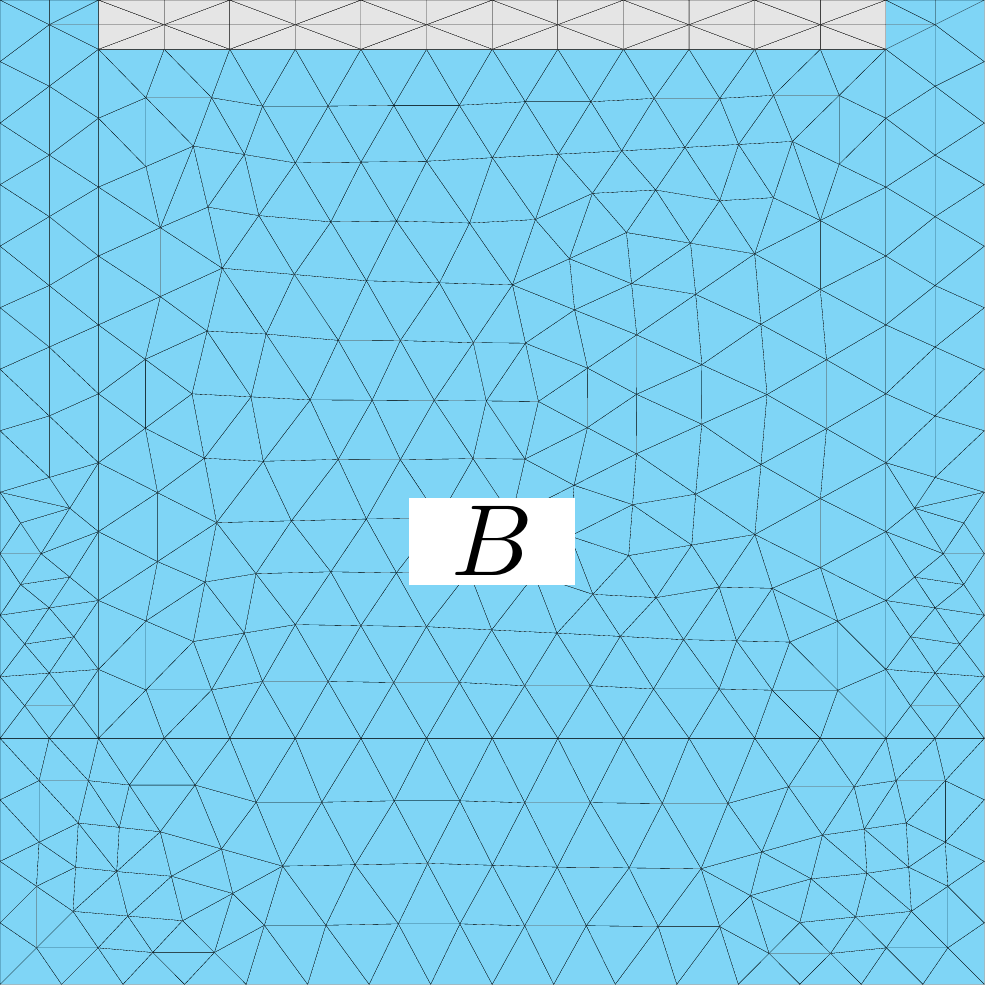}}
\caption{
Subdomains \cref{eq:subdomains-convex} for the numerical experiments in \cref{ssection:numexp_stab_tuning}-\cref{ssection:numexp_pollution}.
}
\label{fig:subdomains-convex}
\end{figure}

\subsection{Tuning of stabilization parameters}\label{ssection:numexp_stab_tuning}

Even though the convergence rates in \cref{thm:L2-error-perturbed-data} hold for any finite $\gammaCIP_1,\gammaGLS,\alpha >0$, optimizing the stabilizing parameters 
can have a significant impact on the quality of the obtained numerical solution and the stability of the linear systems.
For simplicity we will set the non-essential stabilization parameters, i.e.\ $\beta_j$ for $ j \geq 1 $ and $ \gammaCIP_{j}$ for $ j \geq 2 $, to zero in all numerical experiments to 
follow except for \cref{ssection:numexp_pollution} where their potential benefits are investigated.  
Here we will optimize for the remaining essential parameters. \par  
\begin{figure}[htbp]
\centering
\includegraphics[width=\textwidth]{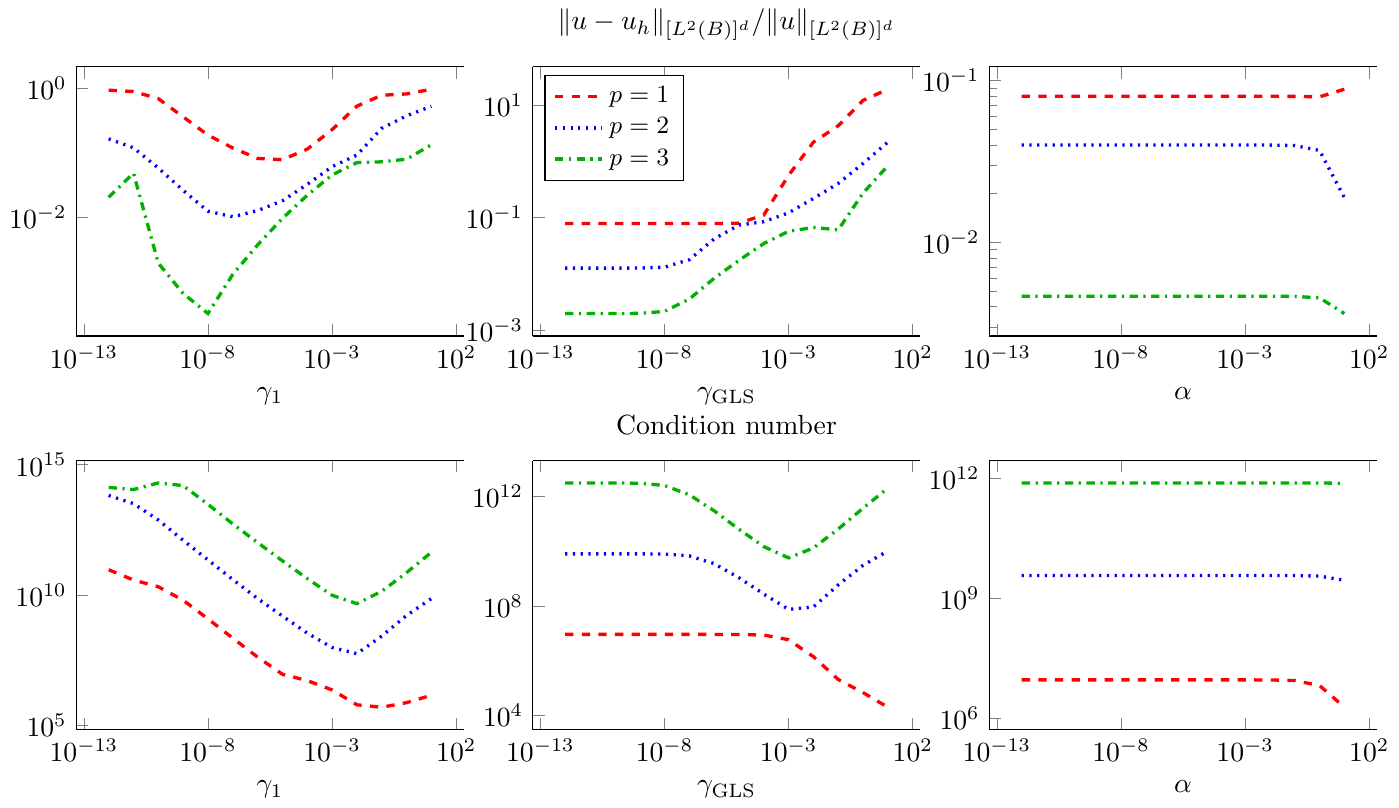}
\caption{
	Relative $L^2$-error $\norm{u-u_h}_{ B } / \norm{u}_{ B } $  in $B$ for geometrical setup of \cref{fig:subdomains-convex} and 
	oscillatory reference solution given in \cref{eq:oscillatory-refsol} for different stabilization parameters. 
	For each of these plots only one penalty parameter has been varied, while the other parameters remain fixed. For the left column we set $\gammaGLS= 10^{-12}, \alpha = 10^{-3}$, 
	for the middle column $\gammaCIP_1 = 10^{-5}/p^{3.5}, \alpha = 10^{-3} $ and for the right column $\gammaCIP_1 =  \gammaGLS =  10^{-5}/p^{3.5}$.
	}
\label{fig:convex-oscillatory-stabsweep-ill-posed-k6}
\end{figure}
The Lam\'{e} coefficients for this experiment will be chosen as 
\begin{equation}\label{eq:Lame_var}
\mu = 1 + \frac{1}{2} \sin(x) \sin(y), \qquad 
\lambda = 1.25 + \frac{1}{2} \cos(x) \cos(y).
\end{equation}
The right hand side $f$ is manufactured so that the exact solution of the problem is given by 
\begin{equation}\label{eq:oscillatory-refsol}
 u(x,y) = \sin(k \pi x) \sin( k \pi y)
	 \begin{pmatrix}
	  1 \\
	  1
\end{pmatrix}
.
\end{equation}
The dependence of the relative errors $\norm{u-u_h}_{ B } / \norm{u}_{ B } $ and the condition number of the system matrix on the penalty parameters is displayed in \cref{fig:convex-oscillatory-stabsweep-ill-posed-k6} for $k=6$ on a fixed mesh which is obtained by two consecutive refinements of the initial mesh shown in \cref{fig:omega-Ind-convex-level0}.
Firstly, it can be noticed that the error in terms of $\gammaCIP_1$ behaves like a well with an approximate minimum at $\gammaCIP_1 = 10^{-5}/p^{3.5}$.
The error invariably has to increase as $\gammaCIP_1$ goes to zero since we loose control over the condition number of the linear system.
We now fix $\gammaCIP_1 = 10^{-5}/p^{3.5}$ and show the behavior of the error as $\gammaGLS$ varies in the central plot of \cref{fig:convex-oscillatory-stabsweep-ill-posed-k6}.
It seems that $\gammaGLS$ basically has to be chosen sufficiently small.
However, let us mention that if $\gammaCIP_1$ was chosen smaller, e.g. $\gammaCIP_1=10^{-12}$, then the error would also exhibit a well-like structure similar as shown in the left column of \cref{fig:convex-oscillatory-stabsweep-ill-posed-k6}.
From now on it seems then appropriate to set $\gammaGLS  =  \gammaCIP_1$. 
The right plot of \cref{fig:convex-oscillatory-stabsweep-ill-posed-k6} shows that the Tikhonov parameter $\alpha$ has almost no influence on the error or the condition number.
We will set $\alpha = 10^{-3}$ from now on. \\

\begin{figure}[htbp]
\centering
\includegraphics[width=.5\textwidth]{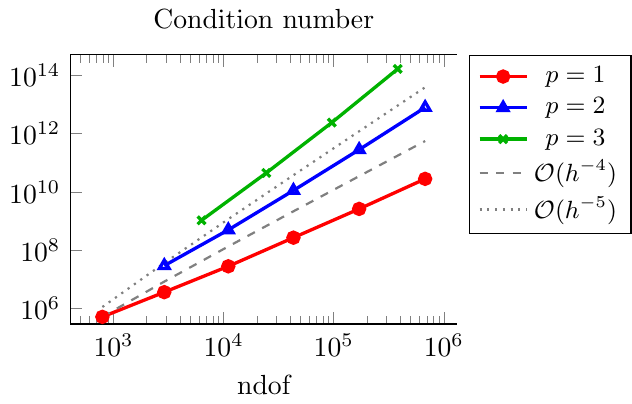}
\caption{
Condition number of the system matrix for $k=6$ in terms of the number of degrees of freedom, respectively the mesh width.
}
\label{fig:Cond-k6-convex}
\end{figure}

As shown in \cref{fig:convex-oscillatory-stabsweep-ill-posed-k6} the condition number of the linear systems is already very high on a moderately refined mesh and 
appears to scale unfavourably with $p$.
This is investigated further in \Cref{fig:Cond-k6-convex} which displays the condition number of the linear systems for our choice of penalty parameters under mesh refinement. 
An approximate scaling of $\mathcal{O}(h^{-2.5-p})$ is observed. 
Hence, when using higher polynomial orders one is more likely to encounter ill-conditioning effects.
This issue should be kept in mind when analyzing the numerical results on fine meshes, in particular for orders $p>1$.

\subsection{Data perturbations}\label{ssection:numexp_data_perturb}

With the stabilization parameters determined as above let us now proceed to the numerical verification of \cref{thm:L2-error-perturbed-data}.
The exact solution, respectively the exact data $u_{\omega}$ and $f$, will be chosen as in \cref{ssection:numexp_stab_tuning}, but we will 
assume now that only perturbed data
\[  \tilde{u}_{\omega} := u_{\omega} + \delta u, \quad  
    \tilde{f} := f + \delta f 
\]
with random perturbations 
\[
\norm{\delta u}_{ \omega } = \mathcal{O}\left(h^{p-\theta}\right), \quad 
\norm{\delta f}_{ \Omega } = \mathcal{O}\left(h^{p-\theta}\right),
\]
for some $\theta \in \mathbb{N}_0$ is available for implementing our method. 
According to \cref{thm:L2-error-perturbed-data} we have the error bound 
\begin{equation}\label{eq:numexp_perturb_error_bound} 
\norm{u - u_h}_{ B } \leq C   h^{\tau p-\theta} \left( 1 + \norm{u}_{ [H^{p+1}(\Omega)]^d } \right),
\end{equation}
which means that achieving convergence requires the condition $\tau p-\theta > 0$. 
\begin{figure}[htbp]
\centering
\includegraphics[width=\textwidth]{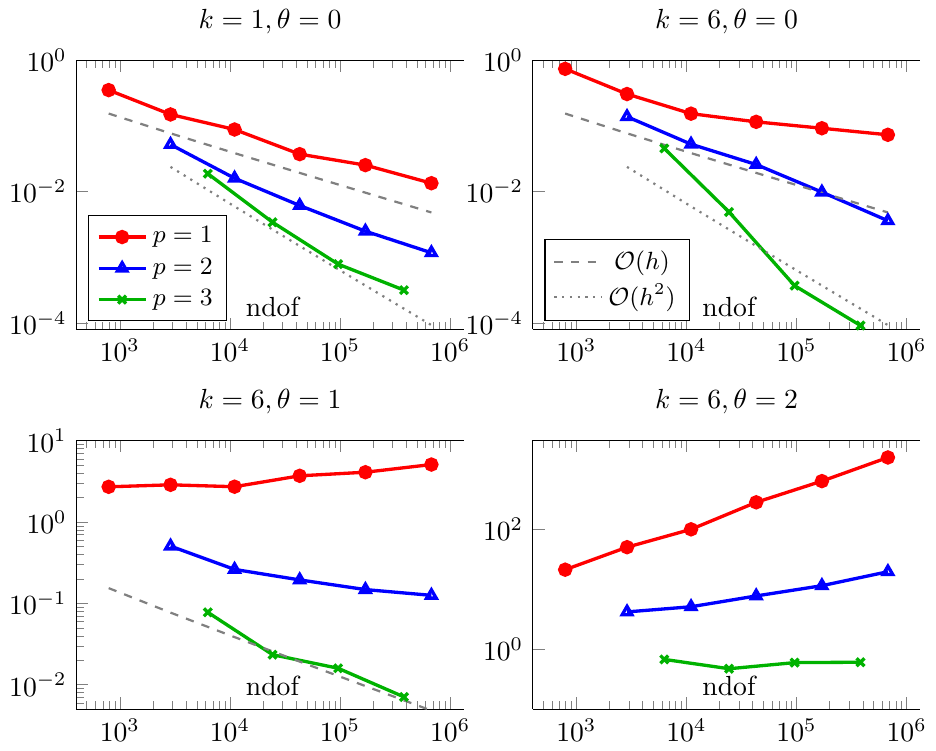}
\caption{
Relative error $\norm{u-u_h}_{ B } / \norm{u}_{ B } $ for geometrical setup shown in \cref{fig:subdomains-convex} in terms of the strength of the data perturbation.
}
\label{fig:Convex-Oscillatory-ill-posed-perturb}
\end{figure}
The relative errors for $\theta=0$ for two different wavenumbers $k=1$ and $k=6$ are compared in the first row of \cref{fig:Convex-Oscillatory-ill-posed-perturb}.
For both cases the observed convergence rates are consistent with \cref{thm:L2-error-perturbed-data}. 
As one may expect, for a smooth reference solution higher polynomial orders deliver higher accuracy with fewer degrees of freedom compared to piecewise affine linear elements.
Note that the errors are in general higher for $k=6$ than for $k=1$. The case $p=3$ is a lucky exception.
The dependence of the error on the wavenumber will be investigated more thoroughly in \cref{ssection:numexp_pollution}.
Let us now turn to the discussion of the second row of \cref{fig:Convex-Oscillatory-ill-posed-perturb} which displays the results for stronger perturbations.
According to \cref{eq:numexp_perturb_error_bound}, we would expect the $p=1$ method to diverge for $\theta = 1$ which is confirmed by \cref{fig:Convex-Oscillatory-ill-posed-perturb}.
The $p=2$ method still converges, albeit at an extremely slow rate, whereas the $p=3$ method at least manages to converge linearly. 
Based on this result it is consistent that for $\theta=2$ convergence is no longer observed for any $p \leq 3$ as shown in the lower right plot of \cref{fig:Convex-Oscillatory-ill-posed-perturb}.  

\subsection{Pollution error}\label{ssection:numexp_pollution}
In this section the dependence of the error on the wavenumber $k$ will be investigated. 
To this end, we stick to the setup of the previous two subsections. 
Based on results for the CIP-FEM applied to well-posed Helmholtz equations available in the literature \citep{W13,ZW13,DH15,ZW22}, 
one would expect a scaling of
\begin{equation}\label{eq:k-norm-scaling}
  k \norm{u-u_h}_{ B } + \norm{ \nabla u- \nabla u_h}_{ B }  \sim k 
\end{equation}
as $k$ increases when $kh < 1$ remains constant. 
To connect to these results, note that the method proposed in this article can also be applied in the setting in which boundary data is available on $\partial \Omega$.
Here we assume that Dirichlet data on the whole boundary is given which we implement in a strong sense (alternatively, a weak imposition following the technique of Nitsche could be used).
We will denote this as the ``well-posed'' problem to the distinguish it from the ``ill-posed'' problem we usually consider in this paper.
\par 

\begin{figure}[htbp]
\centering
\includegraphics[scale=1.15]{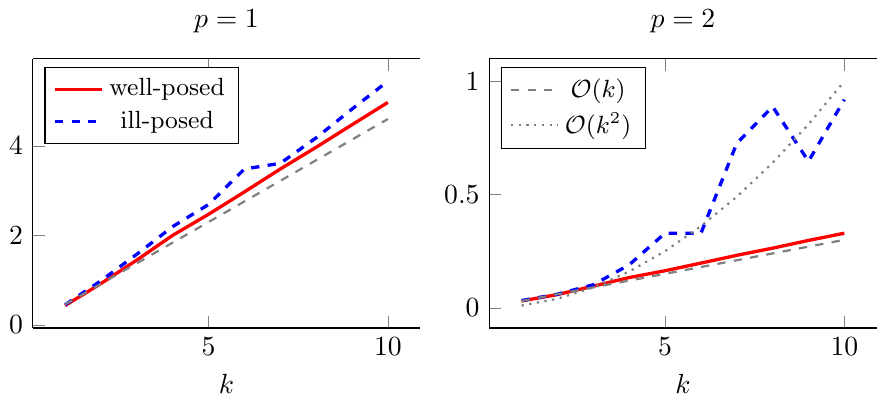}
\caption{The weighted error $ k \norm{u-u_h}_{ B } + \norm{ \nabla u- \nabla u_h}_{ B }  $ under mesh refinement for constant $kh$ using $\beta_j=0$ for $ j \geq 1 $. 
}
\label{fig:Convex-Oscillatory-kh-scaling-knorm}
\end{figure}

In \cref{fig:Convex-Oscillatory-kh-scaling-knorm} the weighted error in \cref{eq:k-norm-scaling} is then displayed for these two different settings. 
For this experiment the case of unperturbed data was considered.
For order $p=1$ a linear scaling in $k$ is observed regardless of whether the well-posed or ill-posed problem is considered. 
However, for $p=2$ the error for the ill-posed problem grows significantly faster than linear and appears to be unstable. 
This could possibly also be related to ill-conditioning of the linear systems, cp.\ \cref{ssection:numexp_stab_tuning}.
Similar results are observed for $p=3$. \par  
Let us now investigate if this phenomenon can be mitigated by choosing $\beta_j > 0$, which amounts to adding the stabilisation term $s_{\beta}$, see equation \cref{eq:s_beta-def}, to the Lagrangian.
The results displayed in \cref{fig:Convex-Oscillatory-kh-scaling-knorm-perturb} show that by choosing $\beta_2$ large enough a linear scaling of the weighted error for $p=2$ 
can indeed be achieved. 
However, a comparison of the middle and right plot of \cref{fig:Convex-Oscillatory-kh-scaling-knorm-perturb} shows that this comes at the expense of increasing the overall error 
in the ill-posed case significantly, which is especially noticeable for lower wavenumbers. 
Note that this behavior does not appear in the well-posed case (the red lines in the middle and right plots of \cref{fig:Convex-Oscillatory-kh-scaling-knorm-perturb} are nearly identical) which has been run using exactly the same stabilization parameters.
Similar results are obtained for $\beta_j < 0$, i.e.\ only the magnitude of $\beta_j$ matters.  
We conclude by recording that care has to be taken when applying higher order methods to the ill-posed elastodynamics problem as the wavenumber increases.
Further research is required to find satisfactory remedies for this issue.

\begin{figure}[htbp]
\centering
\includegraphics[width=\textwidth]{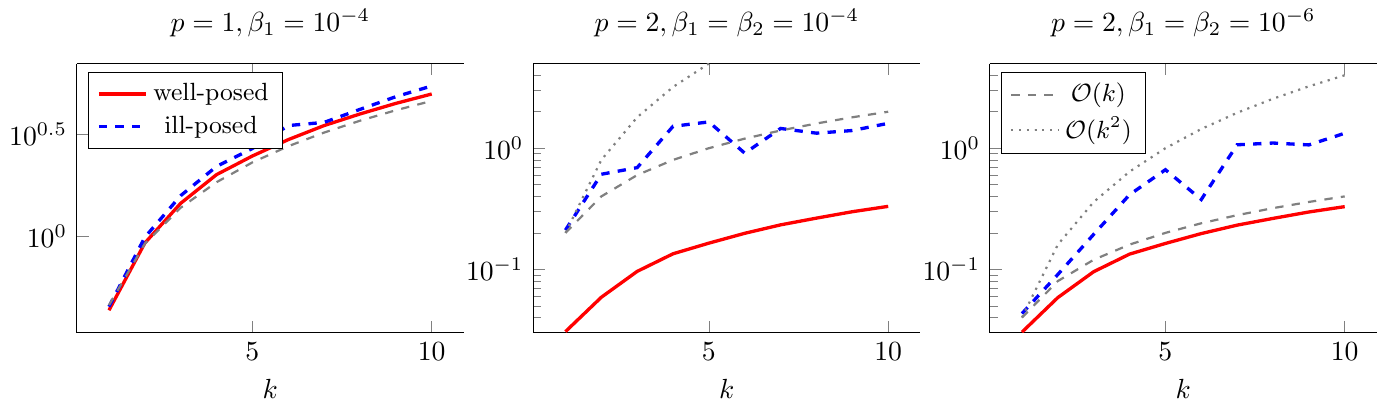}
\caption{The scaled error $ k \norm{u-u_h}_{ B } + \norm{ \nabla u- \nabla u_h}_{ B }  $ under mesh refinement for constant $kh$ utilizing the 
additional stabilization term $s_{\beta}$ defined in \cref{eq:s_beta-def}.  
}
\label{fig:Convex-Oscillatory-kh-scaling-knorm-perturb}
\end{figure}

\subsection{Influence of the geometry}\label{ssection:numexp_split_geom}
It is well-known, see e.g.\ \citep{BNO19,N20,BDE21}, that the geometry of the data and target sets has a major influence on the quality of the reconstruction outside the data domain. 
Roughly speaking, the best results can be expected if the target set $B$ is part of the convex hull of the data set $\omega$ as in the setup shown in \cref{fig:subdomains-convex}.
To increase the level of difficulty, let us now shrink the data set to 
\begin{equation}\label{equation:omega_xi} 
\omega = [0,0.1] \times [0,\xi] \cup [0.9,1.0] \times [0,\xi] \cup [0.1,0.9] \times [0,0.25] 
\end{equation} 
for $\xi = 0.6$. 
This splits the target domain into two halves 
\begin{equation}\label{equation:B_xi}
B_- := [0,1] \times [0,\xi], \qquad
B_+ := [0.1,0.9] \times [\xi,0.95],  
\end{equation}
where $B_-$ is in the convex hull of $\omega$ while $B_+$ is not. 
A sketch of the geometrical setup is given in \cref{fig:split_data_domain_geom}. 
\begin{figure}[htbp]
\centering
\subfloat[ $\omega$ on coarsest mesh.]%
{\label{fig:omega-splitgeom}
\includegraphics[width=.4\textwidth]{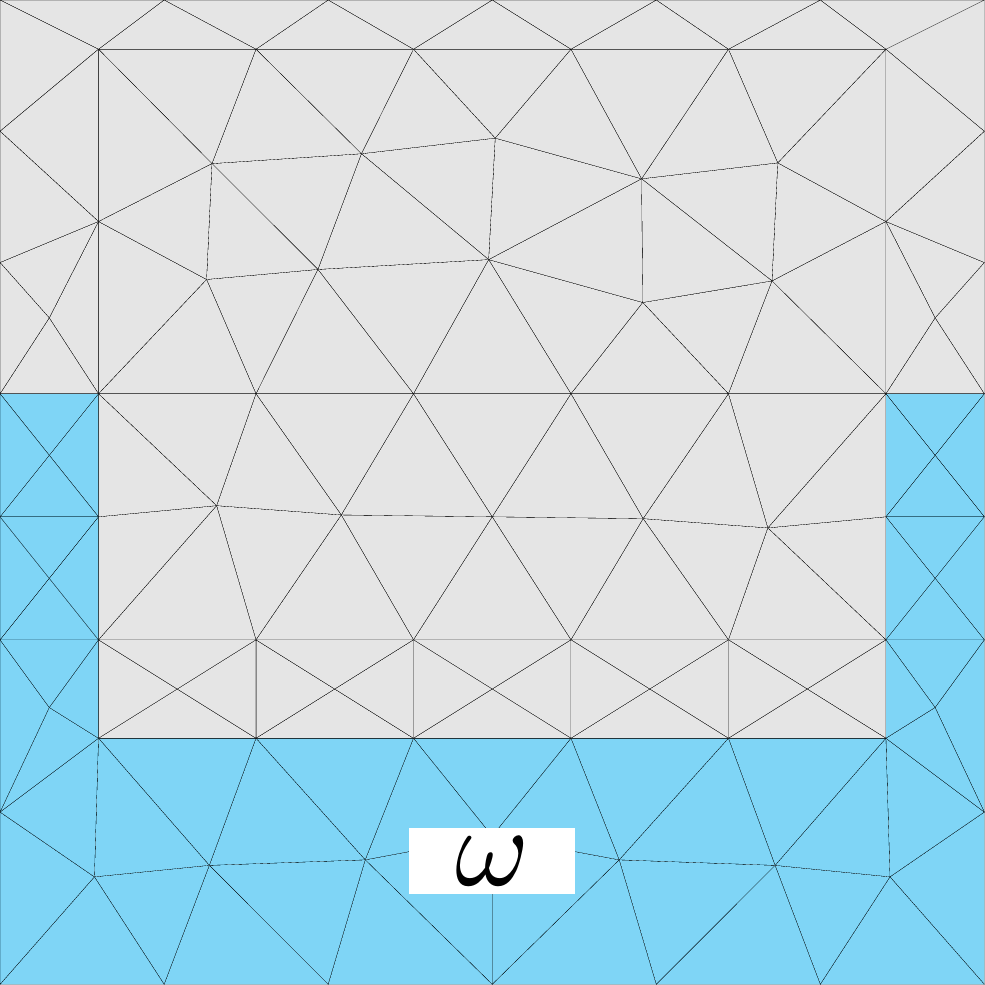}}
\centering
\subfloat[ $B_{\pm}$ on refined mesh.]
{ \label{fig:B-splitgeom} %
\includegraphics[width=.4\textwidth]{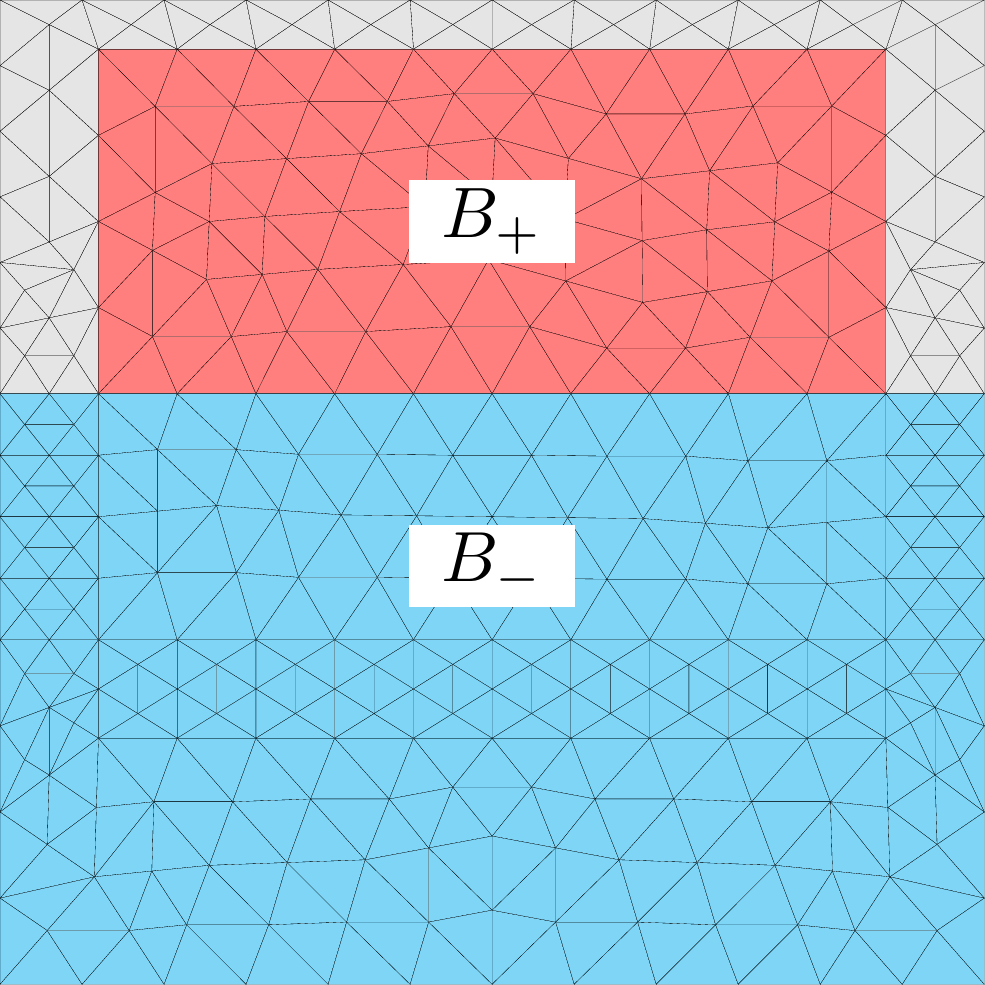}}
\caption{
The data and target domains for the geometry defined in \cref{equation:omega_xi}-\cref{equation:B_xi}.
	}
\label{fig:split_data_domain_geom}
\end{figure}
Let us consider constant coefficients $\mu = 1$ and $\lambda = 1.25$ throughout the entire domain and $k=1$ to render the remainder of the problem as simple as possible. 
The relative errors (using unperturbed data) in the two subdomains $B_{\pm}$ are displayed in \cref{fig:Split-domain-ill-posed-k1-div} as solid lines. 
A stark contrast can be observed (note the different scalings of the vertical axis). 
While near optimal rates of $\mathcal{O}(h^p)$ are obtained in $B_{-}$, we have to use $p=3$ to reach linear convergence rates in $B_{+}$.
This is a clear indication that the conditional stability, in particular the value of the exponent $\tau$ in \cref{eq:cond-stability}, is very sensitive 
to the geometry of the sets $\omega$ and $B$. 
\subsubsection{Adding additional information on divergence of wave diplacement} 
Let us check if these results can be improved if more a priori information is provided. 
Now we will assume that not only $u$ is given in $\omega$ as data, but additionally $\nabla \cdot u = q $ is available in the entire domain $\Omega$. 
This basically means that the divergence part of the stress tensor $\sigma(u)$ in \cref{def:Lu} is known.
The proposed method can easily be modified to cover this case by adding $\frac{1}{2} \norm{ \nabla \cdot u_h - q }_{ \Omega }^2$ as an additional term to the Lagrangian in \cref{eq:Lagrangian}. 
The relative $L^2$-errors for running the same problem as above are displayed as dashed lines in \cref{fig:Split-domain-ill-posed-k1-div}. 
Even though a significant decrease in the absolute value of the errors is observed, the asymptotic convergence rates improve only marginally. 
Hence, additional information on the divergence does apparently not enhance the conditional stability of the problem.

\begin{figure}[htbp]
\centering
\includegraphics[width=\textwidth]{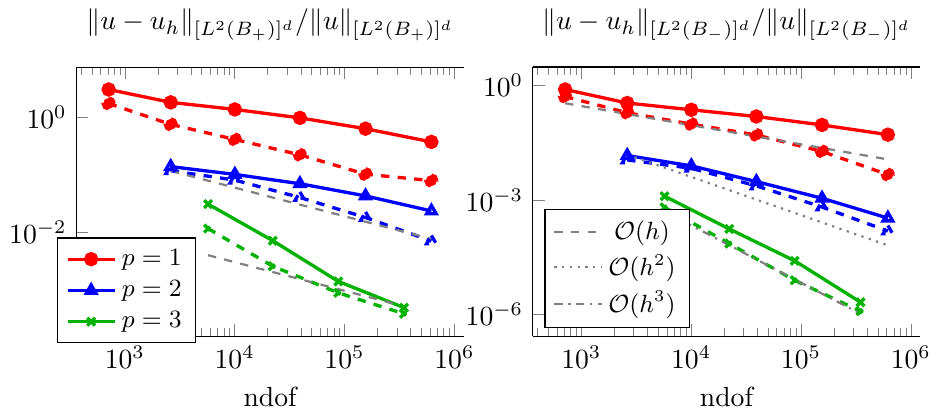}
\caption{ 
The solid lines display the relative $L^2$-errors in the two different parts of the target domain: $B_{-}$ is contained in the convex hull of the data domain while
$B_{+}$ is outside of it.
The dashed lines show the same quantities when additional information on $ \nabla \cdot u$ in $\Omega$ is included in the Lagrangian. 
	}
\label{fig:Split-domain-ill-posed-k1-div}
\end{figure}

\subsection{Jumping shear modulus}\label{ssection:numexp_jump_shear}

\subsubsection{Jump in a plane}\label{ssection:numexp_jump_in_plane} 
Being able to treat Lam\'{e} parameters that exhibit jump discontinuities is of particular interest in applications. 
For example, jumps of $\mu$ and $\lambda$ occur at positions of seismic discontinuities in the Earth's mantle. 
To emulate this behavior in our toy problem, we introduce an artificial interface $\Gamma := \{ (x,y) \in \Omega \mid  y = \eta \}$ and consider a piecewise constant shear modulus 
\bel{eq:mu_pm}
\mu  = 
\begin{cases}
	\mu_+, & \text{for } y > \eta, \\
	\mu_{-} & \text{for } y <  \eta.
\end{cases}
\ee
A weak solution has to fulfill the interface conditions 
\bel{eq:jump_cond_weak}
\jump{u}_{\Gamma} = 0; \quad \jump{\sigma(u) \cdot \mathbf{n}}_{\Gamma} = 0, \quad \text{across } \Gamma.
\ee
Let us denote $\Omega_{+} := \Omega \cap \{ y > \eta \}$ and $\Omega_{-} := \Omega \cap \{ y < \eta \}$. 
\begin{figure}[htbp]
\centering
\includegraphics[width=\textwidth]{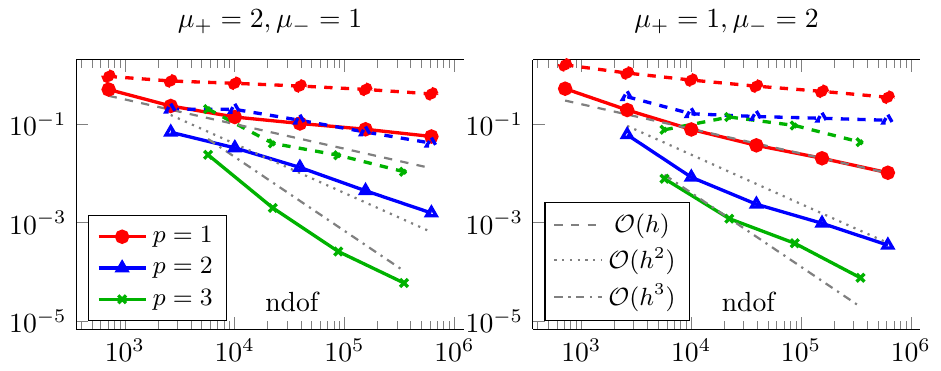}
\caption{Relative errors for geometry shown in \cref{fig:split_data_domain_geom} for a shear modulus which jumps in the plane separating the subdomains $B_{\pm}$.  
We consider $k=4$ and measure the errors in the convex and non-convex part of the target domain separately. The solid lines display $\norm{u-u_h}_{  B_{-} } / \norm{u}_{ B_{-} } $ while the dashed lines show $\norm{u-u_h}_{  B_{+} } / \norm{u}_{ B_{+} } $. }
\label{fig:Jump-Split-domain-k4}
\end{figure}
We make the following ansatz for the wave displacement $u_+ $ in $\Omega_+$ and $u_{-}$ in $\Omega_{-}$: 
\bel{eq:refsol_jump}
u^{+} = \begin{pmatrix} 
	(a_1 + b_1 y + c_1 y^2) \sin(k \pi x) \\
	(a_2 + b_2 y + c_2 y^2) \cos(k \pi x) 
	\end{pmatrix}, 
	\; 
u^{-} = \begin{pmatrix} 
	\sin(k \pi x) \cos( k \pi (y-\eta)) \\  	
	\cos(k \pi x) \cos( k \pi (y-\eta))  	
\end{pmatrix}. 
\ee
As shown in \cref{section:refsol_jump}, the interface conditions of \cref{eq:jump_cond_weak} can be fulfilled by choosing: 
\begin{align}
\begin{aligned}\label{eq:abc-jump}
& b_1 = 0, \qquad   c_1 = \frac{k \pi}{2 \eta} \left[ \frac{\mu_+ - \mu_{-} }{ \mu_{+} }   \right], \qquad a_1 = 1 - c_1 \eta^2, \\ 
&  b_2 = 1, \qquad c_2 = -\frac{1}{2 \eta}, \qquad a_2 = 1 - b_2 \eta - c_2  \eta^2.
\end{aligned}
\end{align}

For the numerical experiment we consider the geometry shown in \cref{fig:split_data_domain_geom} and set $\eta = \xi = 0.6$ so that the jump 
occurs in the plane separating the subdomains $B_{\pm}$ and is respected by the mesh. 
A contrast of about two between $\mu_{+}$ and $\mu_{-}$ is realistic for applications in Earth's seismology, see e.g.\ the reference Earth model of \citep{DA81}.
So we consider $\mu_{\pm} \in \{1,2 \}$ and let $\lambda = 1.25$.  
The relative errors for unperturbed data are displayed in \cref{fig:Jump-Split-domain-k4} for $k=4$.
Similar results as in \cref{ssection:numexp_split_geom} in which a globally constant shear modulus was considered are observed, i.e.\ we achieve rates of nearly 
$\mathcal{O}(h^p)$ in $B_{-}$, whereas the method struggles in $B_+$ but does not break down either. 
By comparing the left and the right plot in \cref{fig:Jump-Split-domain-k4} we notice that doubling the value of $\mu$ apparently leads to a reduction 
of the errors in the respective subdomain. 
This is reasonable since it basically amounts to halving the wavenumber in this subdomain. 
Overall, the presence of a jump appears to have little influence for the considered problem despite the fact that the theoretical error estimate given in \cref{thm:L2-error-unperturbed}
cannot be applied to this case due to insufficient regularity of the shear modulus. 

\subsubsection{Data only at bottom with jump between target domain and exterior} 

Finally, we consider an even more challenging setup shown in \cref{fig:BottomDataJumpIncl-geom} to explore the limits of the proposed method. 
Now data $\omega = [0,1] \times [0,0.25]$ is only available at the bottom of the domain.
The shear modulus is set to $\mu_e$ in the exterior and to $\mu_i$ iside the target domain 
\bel{eq:B_square}
B_- \cup B_+ = [x_L,x_R] \times [y_L,y_R] = [0.25,0.75] \times [0.25,0.9], 
\ee 
which is separated by the plane $ \{ y = 0.6 \}$ into two halves.
Note that in this example the entire target domain is situated outside the convex hull of the data set.
However, $B_-$ is closer to the data set than $B_+$ which should aid the reconstruction. 
We set $\lambda = 1.25$ and use the following reference solution 
\bel{eq:refsol_jump_incl}
u = \zeta^2 \begin{pmatrix} 
	\cos(k \pi x)  \sin(k \pi y)   \\
	\cos(k \pi x)  \cos(k \pi y)  
	\end{pmatrix}
\text{ in }  B_- \cup B_+, \; 
u = \zeta^2 
  \begin{pmatrix}  
   \sin(k \pi x)  \sin(k \pi y)  \\
   \sin(k \pi x)  \cos(k \pi y)    
	\end{pmatrix} 
\text{ else, } 
\ee
for $\zeta := (x-x_L)(x-x_R)(y-y_L)(y-y_R)$.
The results shown in \cref{fig:Bottom-data-incl-k4} for $k=4$ are already so poor even without a jump that we decided 
to lower the wavenumber even further to $k=1$ to investigate whether the effect of a jump can be detected. 
However, the results for $k=1$ shown in \cref{fig:Bottom-data-incl-k1} for different combinations of $\mu_i$ and $\mu_e$ do not provide evidence 
that the presence of a jump is of significant importance here. 
Instead, the major variable appears to be the distance between the data and target domain which accounts for the observation that the errors are about two 
order of magnitudes larger in $B_+$ than in $B_-$. 
Another important factor is the size of the shear in the target domain as already observed in \cref{ssection:numexp_jump_in_plane}.
Finally, we remark that our method also performed well in several further setups featuring jump discontinuities in the shear modulus not shown in this article.

\begin{figure}[htbp]
\centering
\subfloat[ Geometry.]%
{\label{fig:BottomDataJumpIncl-geom}
\includegraphics[scale=0.445]{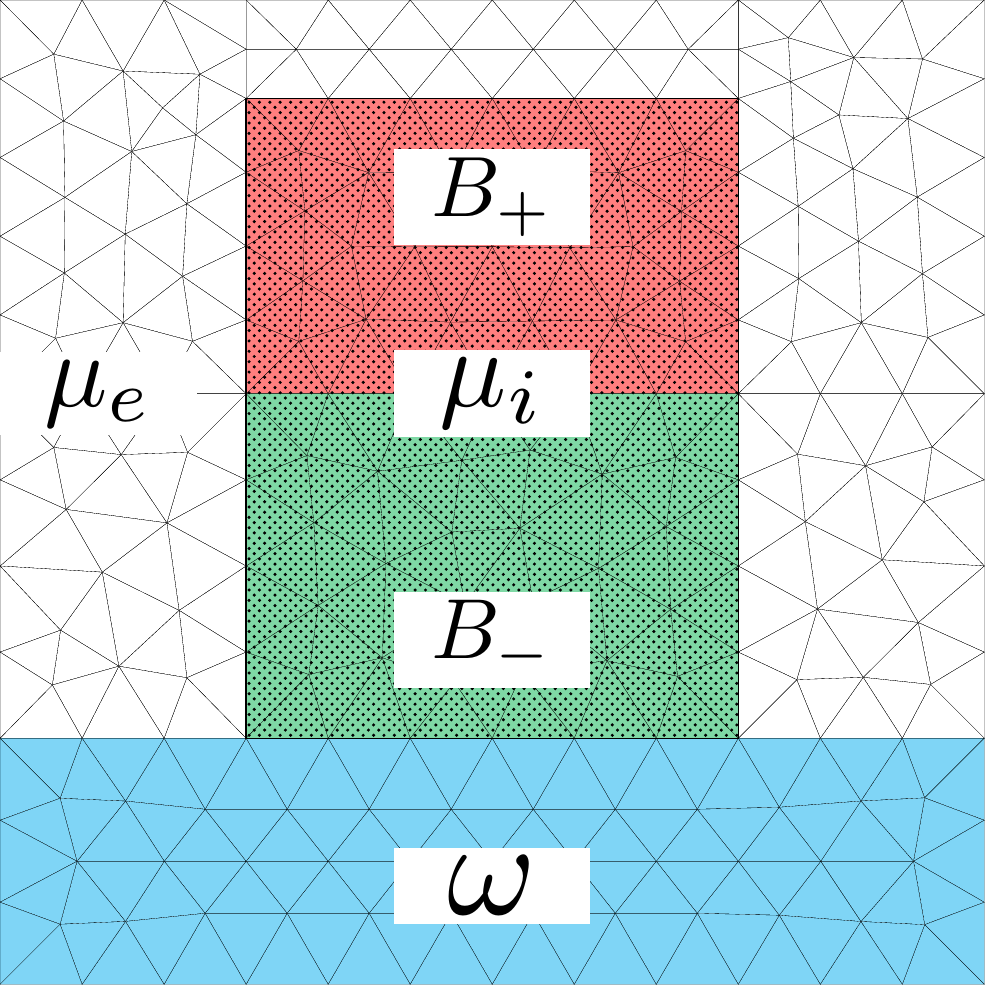}}
\centering
\subfloat[$k=4$.]
{ \label{fig:Bottom-data-incl-k4} 
\includegraphics[scale=1.2]{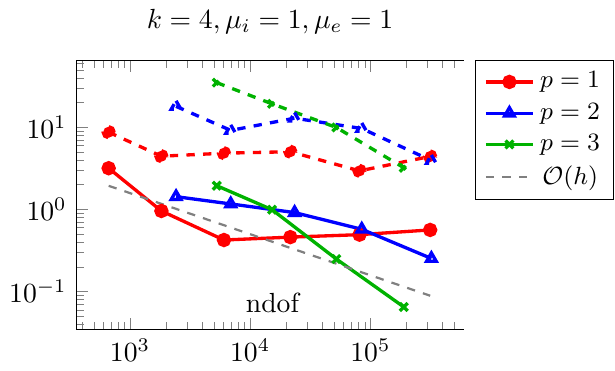}} 
\\ 
\centering
\subfloat[$k=1$.]{ \label{fig:Bottom-data-incl-k1} 
\includegraphics[width=\textwidth]{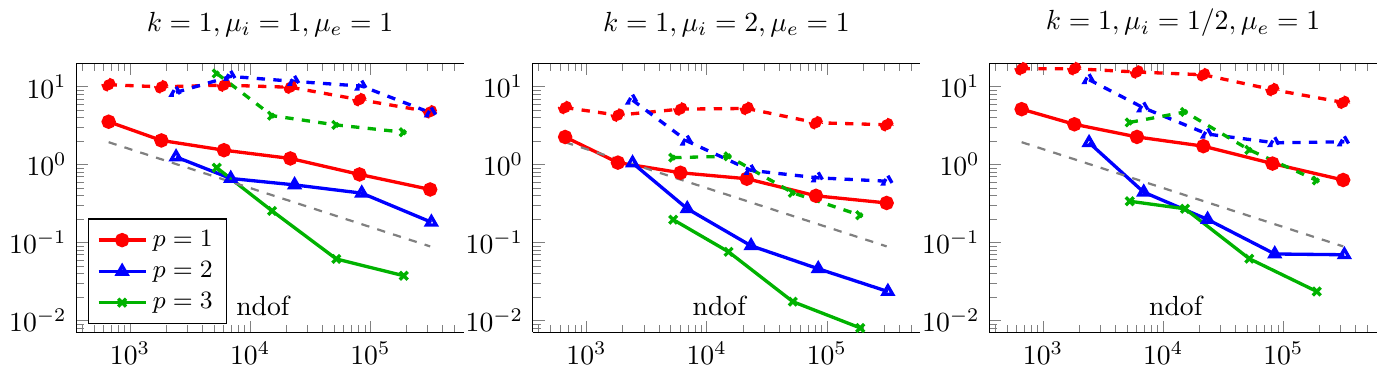}} 
	\caption{ Relative $L^2$-errors $\norm{u-u_h}_{ B_{-}  } / \norm{u}_{  B_{-}   } $ (solid) and  $\norm{u-u_h}_{  B_{+}  } / \norm{u}_{  B_{+} } $ (dashed) for geometry from \cref{fig:BottomDataJumpIncl-geom}. Here, the shear modulus is $\mu_+$ in $B_+ \cup B_-$ and $\mu_-$ in the complement. 
	}  	
\label{fig:BottomDataJumpIncl}
\end{figure}

\section{Conclusion}\label{section:conclusion} 

In this paper we presented a high order stabilized finite element method for unique continuation subject to the Lam\'{e} system. 
The method proceeds by first formulating the data assimilation problem as an ill-posed minimization problem at the discrete level and then adding 
carefully chosen stabilization terms to enhance numerical stability without leading to an exaggerated perturbation of the solution.
Convergence rates have been derived and verified in numerical experiments. 
It turned out that higher order polynomial degrees in the FEM can on the one hand improve efficiency, but on the other hand are at greater risk 
to suffer from ill-conditioning effects. 
We have also observed numerically that the geometry of the data and target domains plays a crucial role for the conditional stability of the problem, which suggests 
that the wavenumber-explicit convergence results proven in \citep{BNO19,N20} for the constant coefficient Helmholtz equation under specific convexity assumptions on the geometry
may extend to the Lam\'{e} system of elastodynamics.
Moreover, in our numerical experiments the geometry appeared to be of far greater importance than the regularity of the Lam\'{e} coefficients as our good numerical results for a discontinuous shear modulus suggest.
However, this point deserves further investigation since it is of course possible that we simply failed to trigger a problematic behavior in our limited set of experiments. 
\par 
Another interesting direction for future research could be to extend our methodology to the reconstruction of the Lam\'{e} parameters from measurements of the wave displacement in the interior of $\Omega$.
Due to its relevance in practical applications, this problem has already attracted significant research interest, see e.g.\ \citep{MZM10,LS17,DBS19}. 
Note also that the potential of the augmented Lagrangian method (on which our approach is to some extent based) for parameter identification problems is well-established, see \citep{IK90,CT03}.

\backmatter





\bmhead{Acknowledgments}
The authors would like to thank Prof. Lauri Oksanen for interesting discussions.
Funding by EPSRC grant EP/V050400/1 is gratefully acknowledged. 
\section*{Declarations}
%

\begin{itemize}
\item \textbf{Funding} This work was funded by EPSRC grant EP/V050400/1.
\item \textbf{Competing interests} The authors have no relevant financial or non-financial interests to disclose.
\item \textbf{Code availability} A docker image containing all software and instructions to reproduce the numerical results in this paper is available from \texttt{zenodo}: \cite{BP22_zenodo}. 
\end{itemize}

\begin{appendices}

\section{Auxiliary proofs for conditional stability estimate}\label{section:Continuum_stability_estimate_derivation}

Proof of \cref{lem:IIP-Robin}:
\begin{proof}
The bound in \cref{eq:stability_IIP_k} with $\norm{f}_{ \Omega }$ on the right hand side 
and a factor of $k$ instead of $k^2$ has been established in \citep[Theorem 2.7]{BG16}.
We follow a standard arguement, see e.g.\ \citep[text between Lemma 3.3 and 3.4]{CM08}  or \citep[proof of Corollary 1.10]{BSW16}, to weaken the norm on the right hand side to $\norm{f}_{\VC^{\prime}}$ at the expense of collecting an additional factor of $k$. \par 
The variational formulation of \cref{eq:PDE-IIP} is given by:   
 \begin{equation}\label{eq:IIP_weak}
	\text{Find } u \in \VC \text{ such that }  b(u,v) = \langle f, v \rangle_{\VC^{\prime} \times \VC  } \qquad \forall v \in \VC, 
\end{equation}
where 
\[
b(u,v) := \int\limits_{\Omega} \left[2 \mu  \sdd(u) : \sdd(\bar{v}) + \lambda \div(u) \div(\bar{v}) - \rho u \bar{v} \right] \dX 
	+ i k \int\limits_{\partial \Omega} u \bar{v} \; \dS. 
\]
Here, the identities $ \sdd(u) : \nabla{\bar{v}} = \sdd(u) : \sdd(\bar{v})$ and  $ \left( \nabla \cdot u \right) I : \nabla \bar{v} = \div(u) \div(\bar{v})$ 
have been employed and we use the notation $\langle \cdot , \cdot \rangle_{ \VC^{\prime} \times \VC }$ for the duality bracket on $\Omega$.
Let us endow $\VC$ with the weighted norm
\[
\norm{u}_{ [H_{k}^1(\Omega)]^d  }^2 := \norm{\nabla u}_{ \Omega }^2 + k^2 \norm{u}_{ \Omega }^2. 
\]
From Korn's second inequality $ \norm{ \sdd(u) }_{ \Omega }^2 \geq \tilde{C}_1 \norm{ \nabla u }_{  \Omega }^2 - \tilde{C}_2 \norm{ u }_{  \Omega }^2 $ and the assumption $\lambda(x) + 2 \mu(x) \geq \delta_0 > 0$ the  G{\aa}rding inequality
\begin{align}
	\Re b(u,u) &= \int\limits_{\Omega} \left( 2 \mu \abs{ \sdd(u) }^2 + \lambda \abs{ \div(u) }^2 - k^2 \abs{u}^2 \right) \; \dX \nonumber \\ 
   & \geq  \int\limits_{\Omega} \left( \min\{\lambda +  2\mu, 2\mu \} \abs{ \sdd(u) }^2 - k^2 \abs{u}^2 \right) \; \dX \nonumber \\ 
	&	\geq C_1 \norm{ \nabla u}^2_{ \Omega }  - (C_2 + k^2) \norm{u}_{ \Omega  }^2 \label{eq:Garding_ieq}
\end{align}
for positive constants $C_1$ and $C_2$ independent of $k$ follows. 
Hence, there exists constants $C^{\prime}$ and $C^{\prime \prime}$ independent of $k$ such that 
\[
b_0(u,v) := b(u,v) + (C^{\prime} + 2k^2)(u,\bar{v})_{ \Omega } 
\]
with $b(\cdot,\cdot)$ as defined in (a), fulfills
\begin{equation}\label{eq:Garding_k_explicit}
\Re b_0(v,v) \geq C^{\prime \prime}  \norm{v}_{ [H_{k}^1(\Omega)]^d  }^2 \quad \forall v \in \VC.
\end{equation}
The Lax-Milgram lemma then implies that the solution $u_0 \in \VC$ of 
\[ b_0(u_0,v) = \langle f, v \rangle  \qquad \forall v \in \VC \]
fulfills $\norm{u_0}_{ H_{k}^1(\Omega) } \leq C  \norm{f}_{ [H_{k}^1(\Omega)]^{\prime} }$ with $C$ being independent of $k$.
The following argument is now based on the observation that the solution $u$ of the original problem given in \cref{eq:IIP_weak} can be split into $u = u_0 + w$, where $w$ solves 
\[ b(w,v) = (C^{\prime} + 2k^2) (u_0,\bar{v})_{ \Omega } \qquad   \forall v \in \VC.  \]
Indeed, as we have 
\[
b(u_0 + w, v) = b(u_0,v) + b(w,v) = b(u_0,v) + (C^{\prime} + 2k^2)(u_0,\bar{v})_{ \Omega } = b_0(u_0,v) = \langle f, v \rangle 
\]
for all $v \in \VC$. 
Since $w$ solves \cref{eq:PDE-IIP} with a right hand side $ (C^{\prime} + 2k^2) u_0 \in [L^2(\Omega)]^d$, we can apply \cite[Theorem 2.7]{BG16} to obtain 
\[ \norm{w}_{ [H_{k}^1(\Omega)]^d } \leq C k^3 \norm{u_0}_{ [L^2(\Omega)]^d } \leq C k^2 \norm{u_0}_{ [H_{k}^1(\Omega)]^d } 
\leq C k^2  \norm{f}_{ [H_{k}^1(\Omega)]^{\prime} },
\]
where the already known bound on $u_0$ has been used. 
Combining both estimates leads to 
\[
\norm{u}_{ [H_{k}^1(\Omega)]^d  }  
\leq C k^2 \norm{f}_{ [H_{k}^1(\Omega)]^{\prime} }
\leq C k^2 \norm{f}_{ \VC^{\prime} },
\]
where the last inequality follows from 
\[
\norm{f}_{ [H_{k}^1(\Omega)]^{\prime} } 
= \sup_{ v \in  [H_{k}^1(\Omega)]^d }  \frac{ \langle f, v \rangle  }{ \norm{v}_{ [H_{k}^1(\Omega)]^d } }  \leq 
\norm{f}_{ \VC^{\prime} }  \sup_{ v \in  [H_{k}^1(\Omega)]^d } \frac{ \norm{v}_{\VC}  }{  \norm{v}_{ [H_{k}^1(\Omega)]^d }  } \leq  
\norm{f}_{ \VC^{\prime} } 
\]
because of $k \geq 1$. 
\end{proof} 

Proof of \cref{cor:cond-stability}:
\begin{proof}
We choose nested balls $B_{R_j}(x_0), j=1,2,3$ and a compactly contained subset $\Omega_1$ of $\Omega$ such that 
\[  B_{R_1}(x_0) \subset B_{R_2}(x_0)  \subset B_{R_3}(x_0) \subset \Omega_1 \subset \Omega. \]
Let $\chi \in C^{\infty}_{0}(\Omega)$ such that $\chi \equiv 1$ on $\Omega_1$. 
By \cref{lem:IIP-Robin} there exists a unique solution $\tilde{u} \in \VC $ of 
\[
\left\{ \begin{array}{rcll} &\mathcal{L} \tilde{u} &=  \chi \mathcal{L}u & \text{ in } \Omega ,\\
&  \sigma(\tilde{u}) \cdot \mathbf{n}_{\partial \Omega} + i k \tilde{u} & = 0 \quad & \text{ on } \partial \Omega,  \end{array}\right.
\]
which fulfills the stability bound $\norm{\tilde{u}}_{\VC} \leq C \norm{ \chi \mathcal{L}u }_{ \VC^{\prime} } $. 
For $\phi \in \VC $ with $\norm{\phi}_{\VC} = 1$ we have $\chi \phi  \in V_{\mathbb{C},0} $ and so 
\[  \langle \chi \mathcal{L} u, \phi \rangle_{ \VC^{\prime} \times \VC  } 
 \leq \norm{ \mathcal{L} u }_{ V_{\mathbb{C},0}^{\prime} } \norm{ \chi \phi  }_{ V_{\mathbb{C},0}  } \leq C \norm{ \mathcal{L} u }_{ V_{\mathbb{C},0}^{\prime} }. 
\] 
Note also that $ \norm{ \mathcal{L} u }_{ V_{\mathbb{C},0}^{\prime} } = \norm{ \mathcal{L} u }_{ V_{0}^{\prime} } $ since 
$\mathcal{L} u $ is real-valued. 
We conclude that $ \norm{\tilde{u}}_{\VC} \leq C  \norm{ \mathcal{L} u }_{ V_{0}^{\prime} } $.
Set $\tilde{w} := u - \tilde{u}$ which fulfills $\mathcal{L} \tilde{w} = 0$ in $\Omega_1$ since $\chi \equiv 1 $ there and we have 
\begin{equation}\label{eq:cond_stab_proof_aux_bound}
\norm{u}_{  B_{R_2(x_0)}  } \leq  \norm{ \tilde{u} }_{  B_{R_2(x_0)}  } + \norm{ \tilde{w} }_{  B_{R_2(x_0)}  }  
\leq C  \norm{ \mathcal{L} u }_{ V_{0}^{\prime} } + \norm{ \tilde{w} }_{  B_{R_2(x_0)}  }. 
\end{equation}
To bound the second term, we can apply \cref{thm:three-ball-ieq-homog} with $\Omega = \Omega_1$ to obtain:
\begin{align*}
\norm{ \tilde{w} }_{  B_{R_2}(x_0) } \leq C 
\norm{ \tilde{w} }_{  B_{R_3}(x_0) }^{1 - \tau } 
\norm{ \tilde{w} }_{  B_{R_1}(x_0) }^{\tau}.
\end{align*} 
We now revert back to $u$ in the right hand side 
\begin{align*}
\norm{ \tilde{w} }_{  B_{R_2}(x_0) } & \leq C 
\left( \norm{ \tilde{u} }_{ \VC  } +  \norm{ u }_{ \Omega }  \right)^{1 - \tau}  
\left( \norm{ \tilde{u} }_{ \VC  } +  \norm{ u }_{  B_{R_1}(x_0) }  \right)^{\tau}  \\ 
& \leq C
\left(  \norm{ \mathcal{L} u }_{ V_{0}^{\prime} }  +  \norm{ u }_{ \Omega }  \right)^{1 - \tau}  
\left(  \norm{ \mathcal{L} u }_{ V_{0}^{\prime} }  +  \norm{ u }_{  B_{R_1}(x_0)  }  \right)^{\tau}.  
\end{align*}
Plugging this bound into \cref{eq:cond_stab_proof_aux_bound} we obtain  
\begin{align*}
\norm{u}_{  B_{R_2(x_0)}  } & \leq C 
\left[  \norm{ \mathcal{L} u }_{ V_{0}^{\prime} }^{1-\tau} \norm{ \mathcal{L} u }_{ V_{0}^{\prime} }^{\tau} + \left(  \norm{ \mathcal{L} u }_{ V_{0}^{\prime} }  +  \norm{ u }_{ \Omega }  \right)^{1 - \tau}  \left(  \norm{ \mathcal{L} u }_{ V_{0}^{\prime} }  +  \norm{ u }_{  B_{R_1}(x_0)  }  \right)^{\tau}   \right]  \\
& \leq C  \left(  \norm{ \mathcal{L} u }_{ V_{0}^{\prime} }  +  \norm{ u }_{ \Omega }  \right)^{1 - \tau}  \left(  \norm{ \mathcal{L} u }_{ V_{0}^{\prime} }  +  \norm{ u }_{  B_{R_1}(x_0)   }  \right)^{\tau}.
\end{align*}
The general case for $\omega \subset B \subset \Omega$ such that $B \setminus \omega$ does not touch the boundary is then obtained 
from the above by applying a covering argument, see \cite[Section 5]{ARRV09} or also \cite{R91}. 
\end{proof}

\section{Reference solution for jumping shear modulus in a plane}\label{section:refsol_jump}

For $d=2$ we have 
\[
\sigma(u) 
= 
\begin{pmatrix}
2 \mu \partial_x u_1 + \lambda \left( \partial_x u_1 + \partial_y u_2 \right)  &   \mu \left( \partial_y u_1 + \partial_x u_2 \right) \\
\mu \left( \partial_x u_2 + \partial_y u_1  \right)      & 2 \mu \partial_y u_2 + \lambda \left( \partial_x u_1 + \partial_y u_2 \right) 
\end{pmatrix}.
\]
Hence, 
\[
\sigma(u) \cdot \mathbf{n} = 
\begin{pmatrix}
\mu \left( \partial_y u_1 + \partial_x u_2 \right)  
 \\
2 \mu \partial_y u_2 + \lambda \left( \partial_x u_1 + \partial_y u_2 \right) 
\end{pmatrix}.
\]
Let us write $u = (u_1,u_2) \in [H^{1}(\Omega)]^2$ and split the components corresponding to the subdomains:
\[
u_1 = 
\begin{cases}
	u_{1}^{+}, & \text{for } y > \eta, \\
	u_{1}^{-}, & \text{for } y <  \eta,
\end{cases}
\qquad 
u_2 = 
\begin{cases}
	u_{2}^{+}, & \text{for } y > \eta, \\
	u_{2}^{-}, & \text{for } y <  \eta.
\end{cases}
\]
The condition $\jump{u}_{\Gamma} = 0$ translates to 
\bel{eq:jump-u-condition}
u_{1}^{+} = u_{1}^{-} \text{ at } y=\eta, \qquad 
u_{2}^{+} = u_{2}^{-} \text{ at } y=\eta. 
\ee

For $\jump{\sigma(u) \cdot \mathbf{n}}_{\Gamma} = 0$ we get 
\bel{eq:jump-sigma-u-condition-1}
\mu_{+} \left( \partial_y u_1^{+} + \partial_x u_2^{+} \right) 
= 
\mu_{-} \left( \partial_y u_1^{-} + \partial_x u_2^{-} \right), 
\quad 
\text{ at } y=\eta. 
\ee
and 
\bel{eq:jump-sigma-u-condition-2}
2 \mu_{+} \partial_y u_2^{+} + \lambda \left( \partial_x u_1^{+} + \partial_y u_2^{+} \right) 
=  2 \mu_{-} \partial_y u_2^{-} + \lambda \left( \partial_x u_1^{-} + \partial_y u_2^{-} \right) \text{ at } y=\eta. 
\ee
In view of \cref{eq:refsol_jump} the conditions \cref{eq:jump-u-condition} become
\bel{eq:jump-u-condition-plugged-in}
a_1 + b_1 \eta + c_1 \eta^2 = 1, \qquad 
a_2 + b_2 \eta + c_2 \eta^2 = 1. 
\ee
Using $\partial_y u_1^{-} = 0$ at $ y=\eta$ the condition \cref{eq:jump-sigma-u-condition-1} leads to
\[
\sin(k \pi x) \left[ \mu_{+} \left( b_1 + 2 c_1 \eta  \right) - \mu_{+} k \pi + \mu_{-} k \pi   \right] = 0,
\]
so we need 
\bel{eq:jump-sigma-u-condition-1-pluggedin}
b_1 + 2 c_1 \eta = k \pi \frac{ ( \mu_{+} - \mu_{-}) }{ \mu_{+}  },
\ee
Plugging \cref{eq:refsol_jump} into \cref{eq:jump-sigma-u-condition-2} we obtain in view of $\partial_{y} u_{2}^{-} = 0$ at $y=\eta$ that 
\begin{align*}
& \left( 2 \mu_{+} + \lambda  \right) \partial_y u_{2}^{+} + \lambda \partial_{x} u_{1}^{+} - \lambda \partial_{x} u_{1}^{-} = 0 \\ 
& \Leftrightarrow 
\left( 2 \mu_{+} + \lambda  \right) (b_2 + 2 c_2 \eta) \cos( k \pi x) + \lambda k \pi \cos(k \pi x) -  \lambda k \pi \cos(k \pi x) = 0.
\end{align*}
Hence,  
\bel{eq:jump-sigma-u-condition-2-pluggedin}
b_2 + 2 c_2 \eta = 0 
\ee
is the condition we need.
By inspection, we see that the choices given in \cref{eq:abc-jump} allow to fulfill  
\cref{eq:jump-u-condition-plugged-in}, \cref{eq:jump-sigma-u-condition-1-pluggedin} and \cref{eq:jump-sigma-u-condition-2-pluggedin}.




\end{appendices}


\bibliography{references}


\end{document}